\documentclass[11pt]{amsart}
\usepackage{amssymb}
\usepackage{amssymb, amsmath}
\usepackage{verbatim}
\let\d=\partial
\let\eps=\varepsilon
\let\wt=\widetilde

\def\cA{{\mathcal A}}

\def\cC{{\mathcal C}}

\def\cO{{\mathcal O}}
\def\cP{{\mathcal P}}
\def\cQ{{\mathcal Q}}

\def\cS{{\mathcal S}}

\def\N{{\mathbb N}}
\def\R{{\mathbb R}}
\def\T{{\mathbb T}}
\def\Z{{\mathbb Z}}

\def\virgp{\raise 2pt\hbox{,}}
\def\cdotpv{\raise 2pt\hbox{;}}
\def\Id{\mathop{\rm Id}\nolimits}

\def\div{ \hbox{\rm div}\,  }

\newcommand{\Frac}{\displaystyle \frac}

\def\ddj{\dot \Delta_j}

\def\na{\nabla}

\def\du{\delta\!u}
\def\dT{\delta\!\Theta}
\def\dU{\delta\!U}

\newcommand{\ef}{ \hfill $ \blacksquare $ \vskip 3mm}

\newtheorem{thm}{Theorem}[section]
\newtheorem{lem}{Lemma}[section]
\newtheorem{rmk}{Remark}[section]

\newtheorem{prop}{Proposition}[section]

\newcommand{\ben}{\begin{eqnarray}}
\newcommand{\een}{\end{eqnarray}}
\newcommand{\beno}{\begin{eqnarray*}}
\newcommand{\eeno}{\end{eqnarray*}}

\begin{document}
\title[The  incompressible limit in $L^p$ type critical spaces]
{The  incompressible limit in $L^p$ type critical spaces}
\author[R. Danchin]{Rapha\"{e}l Danchin}
\address[R. Danchin]{Universit\'{e} Paris-Est,  LAMA (UMR 8050), UPEMLV, UPEC, CNRS, Institut Universitaire de France,
 61 avenue du G\'{e}n\'{e}ral de Gaulle, 94010 Cr\'{e}teil Cedex, France.} \email{raphael.danchin@u-pec.fr}
\author[L. He]{Lingbing He}
\address[L. He]{Department of Mathematical Sciences, Tsinghua University\\
Beijing 100084,  P. R.  China.} \email{lbhe@math.tsinghua.edu.cn}

\begin{abstract} This paper aims at  justifying  the  low Mach number convergence to the incompressible
Navier-Stokes equations   for viscous compressible flows  in the \emph{ill-prepared data} case.
 The fluid domain is either the whole space, or the torus. 

A number of works have been dedicated to this classical issue, all of  them being, to our knowledge, 
related to $L^2$  spaces and to energy type arguments. In the present paper, we investigate the low Mach number convergence  \emph{in the  $L^p$ type critical regularity  framework}.
 More precisely, in the barotropic case,  the divergence-free part  of the initial velocity field just has to be bounded
in the critical Besov space $\dot B^{d/p-1}_{p,r}\cap\dot B^{-1}_{\infty,1}$
for some suitable $(p,r)\in[2,4]\times[1,+\infty].$
We still require $L^2$ type bounds  on the low frequencies of the potential part of the velocity
and on the density, though, an assumption which seems to be unavoidable in the ill-prepared data
framework, because of acoustic waves.

In the last part of the paper, our results are extended to the full Navier-Stokes system for heat conducting fluids.
\end{abstract}
\maketitle

We are concerned with the study of the convergence of the solutions to the 
compressible Navier-Stokes equations when the Mach number $\eps$ goes to $0.$ 
In the barotropic case, the system under consideration reads
$$
 \left\{\begin{array}{l}
\partial_t\rho^\eps+{\rm div}(\rho^\eps u^\eps)=0,\\[0.5ex]\displaystyle
\partial_t(\rho^\eps u^\eps)+{\rm div} (\rho^\eps u^\eps\otimes u^\eps)
-\div\bigl(2\mu(\rho^\eps)D(u^\eps)+\lambda(\rho^\eps)\div u^\eps\Id\bigr)+\frac{\nabla P^\eps}{\eps^2}=0,
\end{array}
\right.\leqno(NSC_\eps)
$$
where $\rho^\eps=\rho^\eps(t,x)\in\R_+$ stands for the  density, 
$u^\eps=u^\eps(t,x)\in\R^d,$  for the  velocity field,  $P^\eps=P(\rho^\eps)\in\R$ is the pressure,
$\lambda=\lambda(\rho^\eps)$ and $\mu=\mu(\rho^\eps)$
are the (given) viscosity functions that are assumed to satisfy $\mu>0$ and $\lambda+2\mu>0.$
Finally, $D(u^\eps)$ stands for the deformation tensor, that is $(D(u^\eps))_{ij}:=\frac12(\d_iu^{\eps,j}+\d_ju^{\eps,i}).$
We assume that the functions $P,$ $\lambda$ and $\mu$ are smooth, 
and we restrict our attention to the case where the fluid domain is either the whole space
$\R^d$ or the periodic box $\T^d$ (combinations such as $\T\times\R^{d-1}$ and so on may be considered as well). 
\smallbreak
At the formal level, in the low Mach number asymptotic, we expect $\rho^\eps$ to tend to some 
constant positive density $\rho^*$ (say $\rho^*=1$ for simplicity) 
and $u^\eps$ to tend to some vector field $v$ satisfying  the  (homogeneous) \emph{incompressible} Navier-Stokes equations: 
$$\left\{\begin{array}{l}\d_tv+v\cdot\nabla v-\mu(1)\Delta v+\nabla\Pi=0,\\[1ex]
\div v=0.
\end{array}
\right.\leqno(NS)
$$
This heuristics has been justified rigorously in different contexts (see e.g. \cite{D3,D4,DH,DG,DGLM,FN-book,HL,Hoff,KM,Klein,KLN,Lions,LM}). 
In the present paper, we want to consider   \emph{ill-prepared data} of the form
   $\rho_0^\eps=\rho^*+\eps a_0^\eps$ and  $u_0^\eps$
   where $(a_0^\eps,u_0^\eps)$ are bounded in a sense that will be specified later on.
Assuming (with no loss of generality) that $P'(\rho^*)=\rho^*=1$ and setting  $\rho^\eps =1+\eps a^\eps,$
we get the following system for $(a^\eps,u^\eps)$:
\begin{equation}
  \left\{
    \begin{aligned}
      & \d_t a^\eps+\frac{\div u^\eps}{\eps}=-\div (a^\eps u^\eps),\\
      &\d_t u^\eps + u^\eps\cdot\nabla u^\eps-\frac{\cA u^\eps}{1+\eps a^\eps} +\frac{\nabla  a^\eps}\eps=\frac{k(\eps a^\eps)}\eps\nabla a^\eps
      \\&\hspace{3cm}+\frac1{1+\eps a^\eps}\div\bigl(2\wt\mu(\eps a^\eps) D(u^\eps)+\wt\lambda(\eps a^\eps)
       \div u^\eps\Id\bigr),
        \end{aligned}
  \right.
  \label{eq:NS}
\end{equation}
where $\cA:= \mu\Delta+(\lambda+\mu)\nabla\div$  with $\lambda:=\lambda(1)$ and $\mu:=\mu(1),$ 
$$k(z):=-\frac{P'(1+z)}{1+z}+P'(1),
\quad\wt\mu(z):=\mu(1+z)-\mu(1)\ \hbox{ and }\ \wt\lambda(z):=\lambda(1+z)-\lambda(1).
$$ 
 In what follows, the exact value of functions $k,$ $\wt\lambda$ and $\wt\mu$ will not matter. We shall only use
 that those functions are smooth and vanish at $0.$
\medbreak
 We strive for \emph{critical regularity assumptions} consistent  with those of  
 the well-posedness issue for the limit system $(NS).$ 
At this stage, let us recall that, by definition,  critical spaces for $(NS)$ are  
norm invariant for all $\ell>0$ by the scaling transformations 
$T_\ell: v(t,x)\longmapsto \ell v(\ell^2t,\ell x),$ in accordance with the 
fact that $v$ is a solution to $(NS)$ if and only if so does $T_\ell v$ 
(provided the initial data has been changed accordingly of course). 

As first observed in \cite{D1}, in the context of the barotropic Navier-Stokes equations \eqref{eq:NS}, 
the relevant scaling transformations  read
\begin{equation}\label{eq:critical}
(a,u)(t,x)\longmapsto (a,\ell u)(\ell^2t,\ell x),\qquad\ell>0,
\end{equation}
which suggest our taking initial data $(a_0,u_0)$ 
in spaces invariant by $(a_0,u_0)(x)\mapsto (a_0,\ell u_0)(\ell x).$ 
\medbreak
In order to be more specific, 
 let us introduce  now the notations and  function spaces that will be used throughout the paper.
 For simplicity, we focus on the $\R^d$ case. Similar notations and definitions may be given 
 in the $\T^d$ case.
 
We are given an homogeneous Littlewood-Paley decomposition
$(\ddj)_{j\in\Z}$ that is a dyadic decomposition in the Fourier space for $\R^d.$
One may for instance set $\ddj:=\varphi(2^{-j}D)$ with $\varphi(\xi):=\chi(\xi/2)-\chi(\xi),$
and $\chi$ a non-increasing nonnegative smooth function supported in $B(0,4/3),$ and with 
value $1$ on $B(0,3/4)$  (see \cite{BCD}, Chap. 2 for more details).

We then define, for  $1\leq p,r\leq \infty$  and $s\in\R,$ the semi-norms
$$
\|z\|_{\dot B^s_{p,r}}:=\bigl\|2^{js}\|\ddj z\|_{L^p(\R^d)}\bigr\|_{\ell^r(\Z)}.
$$
Like in \cite{BCD},  we adopt the following definition of homogeneous
Besov spaces, which turns out to be well adapted to the study of nonlinear PDEs:
$$
\dot B^s_{p,r}=\Bigl\{z\in\cS'(\R^d) : \|z\|_{\dot B^s_{p,r}}<\infty \ \hbox{ and }\ \lim_{j\to-\infty} 
\|\dot S_jz\|_{L^\infty}=0\Bigr\}
\quad\hbox{with }\ \dot S_j:=\chi(2^{-j}D).$$
As we shall work with \emph{time-dependent functions} valued in Besov spaces,
we  introduce the norms:
$$
\|u\|_{L^q_T(\dot B^s_{p,r})}:=\bigl\| \|u(t,\cdot)\|_{\dot B^s_{p,r}}\bigr\|_{L^q(0,T)}.
$$
As  pointed out in \cite{CH},  when  using parabolic estimates in Besov spaces, it is somehow natural
to take the time-Lebesgue norm \emph{before} performing the summation for computing
the Besov norm.  This motivates our introducing the following quantities:
$$
 \|u\|_{\wt L_T^q(\dot {B}^{s}_{p,r})}:= \bigl\|(2^{js}\|\ddj u\|_{L_T^q(L^p)})\bigr\|_{\ell^r(\Z)}.
$$
The index $T$ will be omitted if $T=+\infty$ and we shall denote by $\wt\cC_b(\dot B^s_{p,r})$
 the subset of functions of  $\wt L^\infty(\dot B^s_{p,r})$
which are also  continuous from  $\R_+$ to $\dot B^s_{p,r}$.

Let us emphasize that, owing to Minkowski inequality, we have if  $r\le q$
$$
 \|z\|_{L_T^q(\dot {B}^{s}_{p,r})}\leq \|z\|_{\wt L_T^q(\dot {B}^{s}_{p,r})}
$$
with equality if and only if $q=r.$ Of course, the opposite inequality occurs if $r\geq q.$  
\medbreak
An important  example where those nonclassical norms  are  suitable is the heat equation
\begin{equation}\label{eq:heat0}
\d_tz-\mu\Delta z=f,\qquad z|_{t=0}=z_0
\end{equation}
for which the following family of inequalities holds true (see \cite{BCD,CH}):
\begin{equation}\label{eq:heat}
 \|z\|_{\wt L_T^m(\dot {B}^{s+2/m}_{p,r})}\leq C\bigl(\|z_0\|_{\dot B^s_{p,r}}
 + \|f\|_{\wt L_T^1(\dot {B}^{s}_{p,r})}\bigr)
\end{equation}
for any $T>0,$ $1\leq m,p,r\leq\infty$ and $s\in\R.$
\medbreak
Restricting ourselves to the case of \emph{small and global-in-time} solutions (just for simplicity), 
the reference global well-posedness result for $(NS)$  that we have in mind reads 
as follows\footnote{The statement in the Sobolev framework is due to H. Fujita and T. Kato in \cite{FK}. 
Data in general critical Besov spaces, with a slightly different solution space, have been considered
by H. Kozono and M. Yamazaki in \cite{KY}, and by  M. Cannone, Y. Meyer and F. Planchon in \cite{CMP}. 
The above statement has been proved exactly under this shape by J.-Y. Chemin in \cite{CH}.}:
\begin{thm}\label{thm:NS}
 Let $u_0\in \dot B^{d/p-1}_{p,r}$ with $\div u_0=0$ and $p<\infty,$ and $r\in[1,+\infty].$ 
There exists $c>0$ such that if $$\|u_0\|_{\dot B^{d/p-1}_{p,r}}\leq c\mu$$
then $(NS)$ has a unique global solution $u$ in the space 
$$
\wt L^\infty(\R_+;\dot B^{d/p-1}_{p,r})\cap \wt L^1(\R_+;\dot B^{d/p+1}_{p,r}),
$$
which is also in $\cC(\R_+;\dot B^{d/p-1}_{p,r})$ if $r<\infty.$ Besides, we have
\begin{equation}\label{eq:estins}
 \|u\|_{\wt L^\infty(\dot B^{d/p-1}_{p,r})}+\mu\|u\|_{\wt L^1(\dot B^{d/p+1}_{p,r})}\leq 
 C\|u_0\|_{\dot B^{d/p-1}_{p,r}},
 \end{equation}
for some constant $C$ depending only on $d$ and $p.$
\end{thm}
Although Theorem \ref{thm:NS}  is not related to energy arguments, to our knowledge, 
all the mathematical results proving the convergence of  $(NSC_\eps)$ to $(NS),$ 
strongly rely on the use of $L^2$ type norms in order to get estimates 
independent of $\eps.$  This is  due to the presence of  singular first order
skew symmetric terms (which disappear when performing $L^2$ or $H^s$ estimates)
in  the following linearized  equations of \eqref{eq:NS}:
\begin{equation}
  \left\{
    \begin{aligned}
      & \d_t a^\eps+\frac{\div u^\eps}{\eps}=f^\eps,\\
      &\d_t u^\eps -\cA u^\eps +\frac{\nabla  a^\eps}\eps=g^\eps. 
     \end{aligned}
  \right.
  \label{eq:LNS}
\end{equation}
However, it is clear that those singular terms do not affect the divergence-free part $\cP u^\eps$ of the velocity, 
which just satisfies the heat equation \eqref{eq:heat0}. We thus expect handling $\cP u^\eps$ to be doable 
by means of  a $L^p$ type approach similar to that of  Theorem \ref{thm:NS}. 
At the same time, for low frequencies (`low' meaning small with respect to $(\eps\nu)^{-1}$), 
the singular terms tend to dominate the evolution 
of $a^\eps$ and $\cQ u^\eps,$ which precludes  a $L^p$-type approach with $p\not=2,$ 
as the wave equation is ill-posed in such spaces. Finally,  for very high  frequencies (that is
greater than  $(\eps\nu)^{-1}$), 
it is well known that $a^\eps$ and $\cQ u^\eps$ tend to behave as the solutions of a damped equation and of a heat equation, respectively,  and are thus tractable in $L^p$ type spaces.  Besides,  keeping in mind the
notion of critical space introduced in \eqref{eq:critical},  it is natural to work 
at the same level of regularity for $\nabla a^\eps$ and  $\cQ u^\eps$  (see e.g. \cite{BCD}, Chap. 10, or \cite{CD} for more explanations). 
The rest of the paper is devoted to clarifying this heuristics, first
in the barotropic case (Sections \ref{s:main} to \ref{s:strong}), and next for the full Navier-Stokes-Fourier
system (Section~\ref{s:full}).



\section{Main results} \label{s:main}

Before stating our main results, let us introduce some notation. {}From now on, we agree that for $z\in\cS'(\R^d),$  
\begin{equation}
\label{eq:decompo}  z^{\ell,\alpha}:=\sum_{2^j\alpha\leq 2^{j_0}} \ddj z\quad\mbox{and}\quad z^{h,\alpha}:=\sum_{2^j\alpha>2^{j_0}} \ddj z, \end{equation}
 for some large enough nonnegative integer $j_0$  depending only on $p,$ $d,$ and on the functions $k,$  $\lambda/\nu,$ $\mu/\nu$ with $\nu:=\lambda+2\mu.$
 The corresponding  ``truncated'' semi-norms are defined  as follows: 
\beno \|z\|^{\ell, \alpha}_{\dot B^{\sigma}_{p,r}}:=  \|z^{\ell,\alpha}\|_{\dot B^{\sigma}_{p,r}} 
\ \hbox{ and }\   \|z\|^{h,\alpha}_{\dot B^{\sigma}_{p,r}}:=  \|z^{h,\alpha}\|_{\dot B^{\sigma}_{p,r}}. \eeno

Let $\wt\eps:=\eps\nu.$ Based on the heuristics of the introduction, it is natural  to consider families of data $(a_0^\eps,u_0^\eps)$ so that
\begin{itemize}
\item $(a_0^\eps,\cQ u_0^\eps)^{\ell,\wt\eps} \in\dot B^{d/2-1}_{2,1},$
\item $(a_0^\eps)^{h,\wt\eps}\in \dot B^{d/p}_{p,1},\quad (\cQ u_0^\eps)^{h,\wt\eps}\in \dot B^{d/p-1}_{p,1},$
\item $\cP u_0^\eps\in \dot B^{d/p-1}_{p,r}\cap\dot B^{-1}_{\infty,1}.$
\end{itemize}
Recall that $\dot B^{d/p-1}_{p,r}$ is only embedded in $\dot B^{-1}_{\infty,r}.$
The reason why we prescribe the slightly stronger  assumption $\dot B^{-1}_{\infty,1}$ for $\cP u_0^\eps$ 
is that we need the constructed velocity  to have gradient in  $L^1(\R_+;L^\infty)$
in order to preserve the Besov regularity of $a^\eps$ through the mass equation. 
Indeed, it is well known that for a solution $z$ to the free heat equation, the norm of
$\nabla z$ in $L^1(\R_+;L^\infty)$ is equivalent to that  of $z_0$ in $\dot B^{-1}_{\infty,1}$
(see e.g. \cite{BCD}, Chap. 2).
\medbreak
Our assumptions on the data induce us to look for a solution  to \eqref{eq:NS} in the space $X^{p,r}_{\eps,\nu}$  
of functions $(a,u)$ such that 
\begin{itemize}
\item  $(a^{\ell,\wt\eps} ,\cQ u^{\ell,\wt\eps} )\in \wt\cC_b(\R_+;\dot B^{d/2-1}_{2,1})\cap L^1(\R_+;\dot B^{d/2+1}_{2,1}),$
\item  $a^{h,\wt\eps}\in \wt\cC_b(\R_+;\dot B^{d/p}_{p,1})\cap L^1(\R_+; \dot B^{d/p}_{p,1}),$
\item $\cQ u^{h,\wt\eps}\in\wt\cC_b(\R_+;\dot B^{d/p-1}_{p,1})\cap L^1(\R_+;\dot B^{d/p+1}_{p,1}),$
\item $\cP u\in\wt\cC_b(\R_+;\dot B^{d/p-1}_{p,r}\cap \dot B^{-1}_{\infty,1})\cap \wt L^1(\R_+;\dot B^{d/p+1}_{p,r}\cap \dot B^{1}_{\infty,1})$ (only weak continuity in $\dot B^{d/p-1}_{p,r}$ if $r=\infty$).
\end{itemize}
We shall endow that space  with the norm:
$$\displaylines{
\|(a,u)\|_{X^{p,r}_{\eps,\nu}}\!:=\! \|(a,\cQ u)\|^{\ell,\wt\eps}_{\wt L^\infty(\dot B^{d/2-1}_{2,1})}\!+\!\|\cQ u \|^{h,\wt\eps}_{\wt L^\infty(\dot B^{d/p-1}_{p,1})}
+\|\cP u\|_{\wt L^\infty(\dot B^{d/p-1}_{p,r}\cap \dot B^{-1}_{\infty,1})}
\!+\wt\eps\|a\|^{h,\wt\eps}_{\wt L^\infty(\dot B^{d/p}_{p,1})}
\hfill\cr\hfill+\nu\|(a,\cQ u)\|^{\ell,\wt\eps}_{L^1(\dot B^{d/2+1}_{2,1})}+\nu\| \cQ u\|^{h,\wt\eps} _{L^1(\dot B^{d/p+1}_{p,1})}+\nu\|\cP u\|_{\wt L^1(\dot B^{d/p+1}_{p,r}\cap \dot B^{1}_{\infty,1})}+\eps^{-1}\|a\|^{h,\wt\eps}_{L^1(\dot B^{d/p}_{p,1})}.}
$$

\medbreak
Our main result reads as follows:

\begin{thm}\label{th:main1} Assume that the fluid domain is either $\R^d$ or $\T^d,$ 
that the initial data $(a_0^\eps, u_0^\eps)$ are as above with $1\leq r\leq p/(p-2)$ and that, in addition,
\begin{itemize}
\item  Case $d=2$: $2\leq p<4,$ \smallbreak
\item Case $d=3$: $2\leq p\leq4,$ \smallbreak
\item Case $d\geq4$: $2\leq p<2d/(d-2),$ or $p=2d/(d-2)$ and $r=1.$
\end{itemize}
 Let $\wt\eps:=\eps\nu.$ There exists a constant $\eta$ independent of $\eps$ and of $\nu$ such that  if
\begin{equation}\label{eq:smalldata1}
C_0^{\eps,\nu}:=\|(a_0^\eps,\cQ u_0^\eps)\|^{\ell,\wt\eps}_{\dot B^{d/2-1}_{2,1}}+\|\cQ u_0^\eps\|^{h,\wt\eps}_{\dot B^{d/p-1}_{p,1}}+\|\cP u_0^\eps\|_{\dot B^{d/p-1}_{p,r}\cap \dot B^{-1}_{\infty,1}}+\wt\eps\|a_0^\eps\|^{h,\wt\eps}_{\dot B^{d/p}_{p,1}}\leq \eta \nu,
\end{equation}
then System \eqref{eq:NS} with initial data $(a_0^\eps, u_0^\eps)$ has a global
solution $(a^\eps,u^\eps)$   in the space $X^{p,r}_{\eps,\nu}$ with, for some constant $C$ independent of $\eps$ and $\nu,$
\begin{equation}\label{eq:ue}
\|(a^\eps,u^\eps)\|_{X^{p,r}_{\eps,\nu}}\leq CC_0^{\eps,\nu}.
\end{equation}
In addition, $\cQ u^\eps$ converges weakly to $0$  when $\eps$ goes to $0,$ and, if $\cP u^\eps_0\rightharpoonup v_0$ then  $\cP u^\eps$ converges in the sense of distributions  to the solution of
\begin{equation}\label{eq:ins}
\d_tu+\cP(u\cdot\nabla u)-\mu\Delta u=0,\qquad u|_{t=0}=v_0.
\end{equation}
Finally, if  the fluid domain is $\R^d$ and  $d\ge 3$ then we have
$$\displaylines{
\nu^{1/2}\|(a^\eps,\cQ u^\eps)\|_{\wt L^2(\dot B^{(d+1)/p-1/2}_{p,1})}\leq CC_0^{\eps,\nu}\wt\eps^{1/2-1/p}\quad\hbox{and}\cr
  \|\cP u^\eps-v\|_{\wt L^\infty(\dot B^{(d\!+\!1)/p-3/2}_{p,r})}+\mu\|\cP u^\eps-v\|_{\wt L^1(\dot B^{(d+1)/p+1/2}_{p,r})}\leq 
  C\bigl(\|\cP u_0^\eps-v_0\|_{\dot B^{(d\!+\!1)/p-3/2}_{p,r}}+  C_0^{\eps,\nu}\wt\eps^{1/2-1/p}\bigr).}
$$
In the $\R^2$ case, we have, 
$$\displaylines{
\nu^{1/2}\|(a^\eps,\cQ u^\eps)\|_{\wt L^2(\dot B^{(c+2)/p-c/2}_{p,1})}\leq CC_0^{\eps,\nu}\wt\eps^{\,c(1/2-1/p)}\quad\hbox{and}\cr
  \|\cP u^\eps-v\|_{\wt L^\infty(\dot B^{(c+2)/p-c/2-1}_{p,r})}+\mu\|\cP u^\eps-v\|_{\wt L^1(\dot B^{(c+2)/p-c/2+1}_{p,r})}
  \hfill\cr\hfill \leq C\bigl(\|\cP u_0^\eps-v_0\|_{\dot B^{(c+2)/p-c/2-1}_{p,r}}+  C_0^{\eps,\nu}\wt\eps^{\,c(1/2-1/p)}\bigr)}
$$
 where  the constant  $c$ verifies the conditions $0\leq c\leq 1/2$ and $c<(8-2p)/(p-2)$. 
 \end{thm}
\noindent  Some remarks are in order:
 \begin{enumerate}
 \item  According to \cite{D6}, uniqueness holds true if $r=1$.
We conjecture that it also holds in the other cases but, to the best of our knowledge, 
 the question has not been addressed.
\item  The first part of the  theorem (the global existence issue) may be extended to  $2d/(d+2)\leq p<2$ and all $r\in[1,\infty]$  provided the following smallness condition is fulfilled: 
$$\|(a_0^\eps,\cQ u_0^\eps)\|^{\ell,\wt\eps}_{\dot B^{d/2-1}_{2,1}}+\|\cQ u_0^\eps\|^{h,\wt\eps}_{\dot B^{d/2-1}_{2,1}}+\|\cP u_0^\eps\|_{\dot B^{d/2-1}_{2,r}\cap \dot B^{-1}_{\infty,1}}+\wt\eps\|a_0^\eps\|^{h,\wt\eps}_{\dot B^{d/2}_{2,1}}\leq \eta \nu.$$
Indeed,  Theorem \ref{th:main1} provides a global small solution in $X^{2,r}_{\eps,\nu}.$ Therefore it is only a matter
of checking that the constructed solution has  additional regularity $X^{p,r}_{\eps,\nu}.$
This may be checked by following steps 3 and 4 of the proof below, knowing already that 
the solution is in $X^{2,r}_{\eps,\nu}.$ The condition that $2d/(d+2)\leq p$ comes from 
the part $u^{\ell,\wt\eps}\cdot\nabla a$ of the convection term in the mass equation, as $\nabla u^{\ell,\wt\eps}$
is only in $L^1(\R_+;\dot B^{d/2}_{2,1}),$ and the regularity to be transported is $\dot B^{d/p}_{p,1}.$ 
Hence we need to have $d/p\leq d/2+1$ (see e.g. Chap. 3 of \cite{BCD}). 
The same condition appears when handling $k(\eps a^\eps)\nabla a^\eps$. 

As we believe  the case $p<2$ to be somewhat  anecdotic (it is just  a regularity result), we decided to concentrate on the case $p\geq2$ in the rest of the paper. 
\item 
We can afford  source terms in the mass and velocity equations, the regularity of which is given by 
scaling considerations. 
\item 
 We also expect  results in the same spirit (but only local-in-time) to be provable for large data, 
 as in \cite{D3,D4}.
 \item 
To keep the paper a reasonable size, we also refrained to establish more accurate convergence results
in the case of periodic boundary conditions, based on Schochet's filtering method (see \cite{D4} for 
more details on that issue if $p=2$). 
\end{enumerate}


\section{The proof of global existence for fixed $\eps$ and $\nu$}\label{s:glob}

Recall that  $\nu:=\lambda+2\mu.$ 
Performing the change of unknowns
\begin{equation}\label{eq:change}
(a,u)(t,x):=\eps (a^\eps,u^\eps)(\eps^2\nu t,\eps\nu x)
\end{equation}
and the change of data
\begin{equation}\label{eq:changedata}
(a_0,u_0)(x):= \eps(a^\eps_0,u^\eps_0)(\eps\nu x)
 \end{equation}
reduces the proof of the global existence to the case   $\nu=1$ and $\eps=1.$ 
So in the rest of this section, we assume that $\eps=\nu=1,$ and simply denote 
 \begin{eqnarray}
\label{eq:simdecompo}  &z^{\ell}:=z^{\ell,1}\quad\mbox{and}\quad z^{h}:=z^{h, 1}, \\
\label{simpdecomp}& \|z\|^{\ell}_{\dot B^{\sigma}_{p,r}}:=  \|z^{\ell,1}\|_{\dot B^{\sigma}_{p,r}} \ \hbox{ and }\   \|z\|^{h}_{\dot B^{\sigma}_{p,r}}:= \|z^{h,1}\|_{\dot B^{\sigma}_{p,r}}. \end{eqnarray}
The threshold between low and high frequencies will be set  at $2^{j_0}$ for some large enough nonnegative integer $j_0$
depending only on $d,$ $k,$ $\wt\mu/\nu$ and $\wt\lambda/\nu.$ 
\smallbreak

Resuming to the original variables will yield the desired uniform estimate \eqref{eq:ue}
under Condition \eqref{eq:smalldata1}. Indeed, we notice that we have
up to some harmless constant:
$$\displaylines{
\|(a_0^\eps,\cQ u_0^\eps)\|^{\ell,\wt\eps}_{\dot B^{d/2-1}_{2,1}}+\|\cQ u_0^\eps\|^{h,\wt\eps}_{\dot B^{d/p-1}_{p,1}}+\|\cP u_0^\eps\|_{\dot B^{d/p-1}_{p,r}\cap \dot B^{-1}_{\infty,1}}
+\wt\eps\|a_0^\eps\|_{\dot B^{d/p}_{p,1}}\hfill\cr\hfill=\nu\bigl(\|(a_0,\cQ u_0)\|^\ell_{\dot B^{d/2-1}_{2,1}}
+\|\cQ u_0\|^h_{\dot B^{d/p-1}_{p,1}}+\|\cP u_0\|_{\dot B^{d/p-1}_{p,r}\cap \dot B^{-1}_{\infty,1}}+\|a_0\|_{\dot B^{d/p}_{p,1}}\bigr)}
$$
and
$$
\|(a^\eps,u^\eps)\|_{X^{p,r}_{\eps,\nu}}=\nu\|(a,u)\|_{X^{p,r}_{1,1}}.
$$

\subsection{A priori estimates of the solutions to system  \eqref{eq:NSC}}

In this paragraph, we concentrate on the proof of global estimates for a global smooth 
solution $(a,u)$ to the following system:
\begin{equation}
  \left\{
    \begin{aligned}
      & \d_t a+\div u=-\div (a u),\\
      &\d_t u + u\cdot\nabla u- \wt\cA u+\nabla  a=k(a)\nabla a-J(a)\wt\cA u
      \\&\hspace{4cm}+\frac1{1+a}\div\biggl(2\frac{\wt\mu(a)}\nu D(u)+\frac{\wt\lambda(a)}\nu\div u\Id\biggr),
        \end{aligned}
  \right.
  \label{eq:NSC}
\end{equation}
with $k,$ $\wt\lambda,$ $\wt\mu$ as above, $J(a):= a/(1+ a)$ and $\wt\cA :=\cA/\nu.$
\medbreak
To simplify the presentation we assume  the viscosity coefficients $\lambda$ and $\mu$ 
to be constant (i.e. the last line of the velocity equation in \eqref{eq:NSC}, is zero). The general case will be discussed at the end of the subsection.
\medbreak
Throughout we make the assumption that 
\begin{equation}\label{eq:smalla}
\sup_{t\in\R_+,\, x\in\R^d} |a(t,x)|\leq 1/2
\end{equation}
which will enable us to use freely the composition estimate stated in Proposition \ref{p:compo}. 
Note that as $\dot B^{d/p}_{p,1}\hookrightarrow L^\infty,$ Condition \eqref{eq:smalla} will be ensured by the
fact that the constructed solution has small norm in $X^{p,r}_{1,1}.$

\subsubsection*{Step 1: the incompressible part of the velocity}

Projecting the velocity equation onto the set of divergence free vector fields yields
$$
\d_t\cP u-\wt\mu\Delta\cP u=-\cP\bigl(J(a)\wt\cA u\bigr)-\cP(u\cdot\nabla u)\quad\hbox{with }\ \wt\mu:=\mu/\nu.
$$
Hence, using the estimates \eqref{eq:heat} for the heat equation, we get
\begin{eqnarray}\label{eq:Pu0}
&&\|\cP u\|_{\wt L^\infty(\dot B^{d/p-1}_{p,r})\cap \wt L^1(\dot B^{d/p+1}_{p,r})}\lesssim \|\cP u_0\|_{\dot B^{d/p-1}_{p,r}}
+\|\cP(J(a)\nabla^2 u)+\cP(u\cdot\nabla u)\|_{\wt L^1(\dot B^{d/p-1}_{p,r})}\\\label{eq:Pu1}
&&\|\cP u\|_{\wt L^\infty(\dot B^{-1}_{\infty,1})\cap L^1(\dot B^{1}_{\infty,1})}\lesssim \|\cP u_0\|_{\dot B^{-1}_{\infty,1}}
+\|\cP(J(a)\nabla^2 u)+\cP(u\cdot\nabla u)\|_{L^1(\dot B^{-1}_{\infty,1})}.
\end{eqnarray}
In order to bound the right-hand sides, we use the fact that the $0$-th order Fourier multiplier $\cP$
maps $\wt L^1(\dot B^{d/p-1}_{p,r})$ (or $L^1(\dot B^{-1}_{\infty,1})$) into itself. In addition,   classical product laws and 
Proposition \ref{p:compo}  give (if $p<2d$):
$$
\displaylines{
\|J(a)\nabla^2 u\|_{\wt L^1(\dot B^{d/p-1}_{p,r})}\lesssim \|a\|_{\wt L^\infty(\dot B^{d/p}_{p,1})}\|\nabla^2 u\|_{\wt L^1(\dot B^{d/p-1}_{p,r})},\cr
\|u\cdot\nabla u\|_{\wt L^1(\dot B^{d/p-1}_{p,r})}\lesssim \|u\|_{\wt L^\infty(\dot B^{d/p-1}_{p,r})}\|u\|_{\wt L^1(\dot B^{d/p+1}_{p,r})}.}
$$
Because, by Bernstein inequality,
\begin{equation}\label{eq:B}
\|a\|_{\dot B^{d/p}_{p,1}}\lesssim \|a^\ell\|_{\dot B^{d/p}_{p,1}} + \|a^h\|_{\dot B^{d/p}_{p,1}}
\lesssim 2^{j_0} \|a^\ell\|_{\dot B^{d/2-1}_{2,1}} + \|a^h\|_{\dot B^{d/p}_{p,1}},
\end{equation}
we deduce from \eqref{eq:Pu0} that
\begin{equation}\label{eq:Pua}
\|\cP u\|_{\wt L^\infty(\dot B^{d/p-1}_{p,r})\cap\wt L^1(\dot B^{d/p+1}_{p,r})}\lesssim \|\cP u_0\|_{\dot B^{d/p-1}_{p,r}}
+2^{j_0}\|(a,u)\|_{X^{p,r}_{1,1}}^2.
\end{equation}
Next, in order to bound the r.h.s. of \eqref{eq:Pu1}, we use Bony's decomposition (see \cite{Bony} and the definition in appendix):
$$
J(a)\nabla^2u=T_{J(a)}\nabla^2u+T_{\nabla^2 u}J(a)+R(J(a),\nabla^2u)
$$
and, with the summation convention over repeated indices,
\begin{equation}\label{eq:Puu}
(u\cdot\nabla u)^i=T_{u^j}\d_ju^i+T_{\d_ju^i}u^j+\d_jR((\cP u)^j,u^i)+R((\cQ u)^j,\d_ju^i)\quad\hbox{with }\ 
i=1,\cdots,d.
\end{equation}
On the one hand, $T$ maps $L^\infty\times\dot B^{-1}_{\infty,1}$
and $\dot B^{-1}_{\infty,1}\times L^\infty$ in $\dot B^{-1}_{\infty,1}$
while $R$ maps $\dot B^{d/p}_{p,1}\times\dot B^{d/p-1}_{p,\infty}$ in $\dot B^{d/p-1}_{p,1},$ if $2\leq p<2d.$
Hence, taking advantage of  functional embeddings (adapted to $\wt L^m(\dot B^s_{p,r})$ spaces),  
\begin{equation}\label{eq:Pu2}
\|J(a)\nabla^2u\|_{L^1(\dot B^{-1}_{\infty,1})}\lesssim \|a\|_{\wt L^\infty(\dot B^{d/p}_{p,1})}\|\nabla^2u\|_{\wt L^1(\dot B^{d/p-1}_{p,r})}.
\end{equation}
On the other hand, thanks to the fact that, if $2\leq p<2d,$
$$
\begin{array}{lll}
\|T_{u^j}\d_ju^i\|_{L^1(\dot B^{-1}_{\infty,1})}&\lesssim&\|u^j\|_{L^\infty(\dot B^{-1}_{\infty,1})}\|\d_ju^i\|_{L^1(\dot B^0_{\infty,\infty})}\\[1ex]
\|T_{\d_ju^i}u^j\|_{L^1(\dot B^{-1}_{\infty,1})}&\lesssim&\|\d_ju^i\|_{L^\infty(\dot B^{-2}_{\infty,1})}\|u^j\|_{L^1(\dot B^1_{\infty,\infty})}\\[1ex]
\|\d_jR((\cP u)^j,u^i)\|_{L^1(\dot B^{-1}_{\infty,1})}&\lesssim&\|(\cP u)^j\|_{\wt L^\infty(\dot B^{-1}_{\infty,1})}\|u^i\|_{\wt L^1(\dot B^{d/p+1}_{p,r})}\\[1ex]
\|R((\cQ u)^j,\d_ju^i)\|_{L^1(\dot B^{-1}_{\infty,1})}&\lesssim&\|(\cQ u)^j\|_{\wt L^\infty(\dot B^{d/p-1}_{p,1})}\|\d_ju^i\|_{\wt L^1(\dot B^{d/p}_{p,r})},
\end{array}
$$
we  get 
$$
\|u\cdot\nabla u\|_{L^1(\dot B^{-1}_{\infty,1})}
\lesssim \bigl(\|\cP u\|_{\wt L^\infty(\dot B^{-1}_{\infty,1})}+\|\cQ u\|_{\wt L^\infty(\dot B^{d/p-1}_{p,1})}\bigr)
\|u\|_{\wt L^1(\dot B^{d/p+1}_{p,r})}.
$$
Plugging this latter inequality and \eqref{eq:Pu2} in \eqref{eq:Pu1} and using Bernstein inequality, we end up with 
\begin{equation}\label{eq:Pub}
\|\cP u\|_{\wt L^\infty(\dot B^{ -1}_{\infty,1})\cap L^1(\dot B^{1}_{ \infty,1})}\lesssim \|\cP u_0\|_{\dot B^{ -1}_{\infty,1}}
+2^{j_0}\|(a,u)\|_{X^{p,r}_{1,1}}^2.
\end{equation}


\subsubsection*{Step 2: the low frequencies of $(a,\cQ u)$}

Throughout, we set $p^*=2p/(p-2)$ (that is $1/p+1/p^*=1/2$) and
$1/r^*=\theta/r+1-\theta$ with $\theta=p/2-1.$
Because $2\leq p\leq \min(4,2d/(d-2)),$ we have $\max(p,d)\leq p^*,$ and $r^*\in[1,r].$ 
We shall use repeatedly  the following facts, based on straightforward interpolation inequalities~:
\begin{itemize}
\item  The space
$\wt L^\infty(\dot B^{d/p-1}_{p,r})\cap\wt L^\infty(\dot B^{-1}_{\infty,1})$ is continuously embedded in 
$\wt L^\infty(\dot B^{d/p^*-1}_{p^*,r^*}),$ 
\item The space $\wt L^2(\dot B^{d/p}_{p,r})\cap\wt L^2(\dot B^{0}_{\infty,1})$ is continuously embedded in 
$\wt L^2(\dot B^{d/4}_{4,2})$ (here comes that $r\leq p/(p-2)$),
\item  We have $1/r^*+1/r\geq1$  (again, we use  that $r\leq p/(p-2)$),
\item If $p=d^*$ (that is $p=2d/(d-2)$) then $r=1$ by assumption, and thus $r^*=1,$ too.
\end{itemize}
\medbreak
Now, to estimate the low frequencies of $(a,\cQ u),$ we write that
\begin{equation}\label{eq:lfbaro}
  \left\{
    \begin{aligned}
      & \d_t a+\div\cQ u=-\div (a u),\\
      &\d_t\cQ u-\Delta\cQ u+\nabla a=-\cQ( u\cdot\nabla u)-\cQ (J(a)\wt\cA u)+k(a)\nabla a,
        \end{aligned}
  \right.
\end{equation}
and the energy estimates for the barotropic linearized equations (see \cite{BCD}, Prop. 10.23, or \cite{D1}) thus give
$$\displaylines{
\|(a,\cQ u)\|^\ell_{\wt L^\infty(\dot B^{d/2-1}_{2,1})\cap L^1(\dot B^{d/2+1}_{2,1})}
\lesssim \|(a_0,\cQ u_0)\|^\ell_{\dot B^{d/2-1}_{2,1}}
+\|\div (a u) \|^\ell_{L^1(\dot B^{d/2-1}_{2,1})}\hfill\cr\hfill+\| \cQ( u\cdot\nabla u)\|^\ell_{L^1(\dot B^{d/2-1}_{2,1})}
+\| \cQ (J(a)\wt\cA u)\|^\ell_{L^1(\dot B^{d/2-1}_{2,1})}+\|k(a)\nabla a\|^\ell_{L^1(\dot B^{d/2-1}_{2,1})}.}
$$
Let us first bound\footnote{We
do not get anything better by just considering the low frequencies of $\cQ(u\cdot\nabla u).$}  $u\cdot\nabla u$ in $L^1(\dot B^{d/2-1}_{2,1}).$
For that, we use again decomposition \eqref{eq:Puu}. To handle the first term of \eqref{eq:Puu}, 
we just use that (see \cite{BCD}, Chap. 2)
$$
T: \wt L^\infty(\dot B^{d/p^*-1}_{p^*,r^*})\times\wt L^1(\dot B^{d/p}_{p,r})\longrightarrow
L^1(\dot B^{d/2-1}_{2,1}).
$$
This is due to the fact that $1/p+1/p^*=1/2,$ and that either $d/p^*-1<0$ and $1/r+1/r^*\geq1,$ or $d/p^*-1=0$ and $r^*=1.$
As $T: \wt L^\infty(\dot B^{d/p^*-2}_{p^*,r^*})\times\wt L^1(\dot B^{d/p+1}_{p,r})\longrightarrow
L^1(\dot B^{d/2-1}_{2,1}),$ the second term of \eqref{eq:Puu} also satisfies quadratic estimates
with respect to the norm of the solution in $X^{p,r}_{1,1}.$
\smallbreak
Next, because 
$$
R:\dot B^{d/4}_{4,2}\times\dot B^{d/4}_{4,2}\longrightarrow\dot B^{d/2}_{2,1},
$$
we have
$$
\|\d_j R((\cP u)^j,u^i)\|_{L^1(\dot B^{d/2-1}_{2,1})}\lesssim\|\cP u\|_{L^2(\dot B^{d/4}_{4,2})}\| u\|_{L^2(\dot B^{d/4}_{4,2})}.
$$
For the  last  term of \eqref{eq:Puu}, we just have to use that
$$
R:\wt L^\infty(\dot B^{d/p-1}_{p,1})\times\wt L^1(\dot B^{d/p}_{p,r})\longrightarrow L^1(\dot B^{d/2-1}_{2,1})
\quad\hbox{for }\quad p\in[2,4]\cap[2,2d).
$$
Putting all the above informations together, we conclude that 
\begin{equation}\label{eq:Qu1}
\|u\cdot\nabla u\|_{L^1(\dot B^{d/2-1}_{2,1})}\lesssim \|(a,u)\|_{X^{p,r}_{1,1}}^2.
\end{equation}
In order to bound $(\div(au))^\ell,$ we notice that
\begin{equation}\label{eq:Qu2}
\bigl(\div(au)\bigr)^\ell=\bigl(\div(R(a,u)+T_{a}u)\bigr)^\ell+\div T_{u^\ell} a^\ell 
+\bigl(\div\bigl(\dot S_{j_0}u\,\dot\Delta_{j_0+1}a\bigr)\bigr)^\ell.
\end{equation}
Now, the remainder $R$ and the paraproduct $T$  map $\wt L^\infty(\dot B^{d/p^*-1}_{p^*,1})\times\wt L^1(\dot B^{d/p+1}_{p,r})$ 
in $L^1(\dot B^{d/2}_{2,1})$ and we have $\wt L^\infty(\dot B^{d/p-1}_{p,1})\hookrightarrow \wt L^\infty(\dot B^{d/p^*-1}_{p^*,1})$ because $p^*\geq p.$
Hence 
$$
\|\div(R(a,u)+T_{a}u)\|_{L^1(\dot B^{d/2-1}_{2,1})}\lesssim  \|a\|_{\wt L^\infty(\dot B^{d/p-1}_{p,1})}
\|u\|_{\wt L^1(\dot B^{d/p+1}_{p,r})}.
$$
To handle the third term of \eqref{eq:Qu2}, it suffices to use the fact that 
$$
T:L^2(L^\infty)\times L^2(\dot B^{d/2}_{2,1})\longrightarrow L^1(\dot B^{d/2}_{2,1}).
$$
Finally, 
$$\begin{array}{lll}
\|\dot S_{j_0}u\,\dot\Delta_{j_0+1}a\|_{L^1(L^2)}&\!\!\!\leq\!\!\!& \|\dot S_{j_0}u\|_{L^\infty(L^{p^*})}
\|\dot\Delta_{j_0+1}a\|_{L^1(L^p)}\\[1ex]
&\!\!\!\lesssim\!\!\!& 2^{j_0(1-d/p^*)}\|u\|_{\wt L^\infty(\dot B^{d/p^*-1}_{p^*,r})}\bigl(2^{j_0d/p}\|\dot\Delta_{j_0+1}a\|_{L^1(L^p)}\bigr)2^{-j_0d/p}.
\end{array}$$
Hence 
\begin{equation}\label{eq:Qu2a}
2^{j_0d/2}\|\dot S_{j_0}u\,\dot\Delta_{j_0+1}a\|_{L^1(L^2)}\lesssim 2^{j_0} \|u\|_{\wt L^\infty(\dot B^{d/p^*-1}_{p^*,r})}\|a\|_{L^1(\dot B^{d/p}_{p,1})}^h.
 \end{equation}
We can thus conclude that 
\begin{equation}\label{eq:Qu3}
\|\div(au)\|^\ell_{L^1(\dot B^{d/2-1}_{2,1})}\lesssim 2^{j_0}\|(a,u)\|_{X^{p,r}_{1,1}}^2.
\end{equation}
Next, denoting $\cQ^\ell:=\dot S_{j_0+1}\cQ,$ we write that
\begin{equation}\label{eq:Qu4}
\cQ^\ell(J(a)\wt\cA u)=\cQ^\ell\bigl(T_{\wt\cA u} J(a)+ R(\wt\cA u,J(a))\bigr)+T_{J(a)}\Delta\cQ^\ell u
+[\cQ^\ell,T_{J(a)}]\wt\cA u.
\end{equation}
To handle the first two terms, it suffices to notice that 
\begin{equation}\label{eq:Qu4a}
R \ \hbox{ and }\ T \ \hbox{ map }\ \wt L^\infty(\dot B^{d/p^*-1}_{p^*,r^*})\times\wt L^1(\dot B^{d/p}_{p,1})
\ \hbox{ to }\  L^1(\dot B^{d/2-1}_{2,1}),
\end{equation}
and to use Proposition \ref{p:compo}.
Therefore, by virtue of \eqref{eq:B},
$$
\|T_{\wt\cA u} J(a)+ R(\wt\cA u,J(a))\|_{L^1(\dot B^{d/2-1}_{2,1})}\lesssim 2^{j_0}\|(a,u)\|_{X^{p,r}_{1,1}}^2.
$$
For the third term, we just have to use that $T:L^\infty\times\dot B^{d/2-1}_{2,1}\to\dot B^{d/2-1}_{2,1}.$
Finally the commutator term may be handled according to Lemma \ref{l:com}, which ensures 
that\footnote{Recall that $r=1$ if $p^*=d.$} 
$$
\|[\cQ^\ell,T_{J(a)}]\wt\cA u\|_{L^1(\dot B^{d/2-1}_{2,1})}\lesssim \|\nabla J(a)\|_{\wt L^\infty(\dot B^{d/p^*-1}_{p^*,1})}\|\nabla^2u\|_{\wt L^1(\dot B^{d/p-1}_{p,r})}.$$
Hence using embeddings and  composition estimates, we end up with 
\begin{equation}\label{eq:Qu5}
\|\cQ(J(a)\wt\cA u)\|^\ell_{L^1(\dot B^{d/2-1}_{2,1})}\lesssim 2^{j_0}\|(a,u)\|_{X^{p,r}_{1,1}}^2.
\end{equation}
Finally, we decompose $k(a)\nabla a$ as follows:
$$
\bigl(k(a)\nabla a\bigr)^\ell=\bigl(T_{\nabla a}k(a)+R(\nabla a,k(a))\bigr)^\ell+ T_{(k(a))^\ell}\nabla a^\ell
+\bigl(\dot S_{j_0}k(a)\dot\Delta_{j_0+1}\nabla a\bigr)^\ell.
$$
To bound the first two terms, we use again \eqref{eq:Qu4a} and composition estimates. 
For the third term, we use that $T:L^2(L^\infty)\times L^2(\dot B^{d/2-1}_{2,1})\to L^1(\dot B^{d/2-1}_{2,1}).$
For the last term, we proceed as in \eqref{eq:Qu2a} and get
$$
2^{j_0(d/2-1)}\|\dot S_{j_0} k(a)\,\dot\Delta_{j_0+1}\nabla a\|_{L^2}\lesssim 
2^{j_0}\|k(a)\|_{\dot B^{d/p^*-1}_{p^*,1}}\|a\|_{\dot B^{d/p}_{p,1}}^h.
$$
Therefore, by embedding,
$$
2^{j_0(d/2-1)}\|\dot S_{j_0} k(a)\,\dot\Delta_{j_0+1}\nabla a\|_{L^1(L^2)}\lesssim 
2^{j_0}\|k(a)\|_{L^\infty(\dot B^{d/p-1}_{p,1})}\|a\|_{L^1(\dot B^{d/p}_{p,1})}^h.
$$
For bounding $k(a),$ one cannot use directly Proposition \ref{p:compo}  as  it may happen that $d/p-1<0.$
So  we  write
$$
k(a)=k'(0)\,a+a\wt k(a)\quad\hbox{with }\ \wt k(0)=0.
$$
Now, combining   Proposition \ref{p:compo} and product laws in Besov spaces, we get for $2\leq p<2d,$
\begin{equation}\label{eq:comp}
\|k(a)\|_{\dot B^{d/p-1}_{p,1}}\lesssim  (|k'(0)|+\|a\|_{\dot B^{d/p}_{p,1}})\|a\|_{\dot B^{d/p-1}_{p,1}}.
\end{equation}
So finally, 
\begin{equation}\label{eq:Qu6}
\|k(a)\nabla a\|_{L^1(\dot B^{d/2-1}_{2,1})}\lesssim 2^{j_0} (1+\|a\|_{L^\infty(\dot B^{d/p}_{p,1})})\|(a,u)\|_{X^{p,r}_{1,1}}^2.
\end{equation}
Putting together Inequalities \eqref{eq:Qu1}, \eqref{eq:Qu3}, \eqref{eq:Qu5} and \eqref{eq:Qu6},  we conclude that
\begin{equation}\label{eq:lf}
\|(a,\cQ u)\|^\ell_{\wt L^\infty(\dot B^{d/2-1}_{2,1})\cap L^1(\dot B^{d/2+1}_{2,1})}\lesssim \|(a_0,\cQ u_0)\|^\ell_{\dot B^{d/2-1}_{2,1}}
+2^{j_0}(1+\|a\|_{\wt L^\infty(\dot B^{d/p}_{p,1})})\|(a,u)\|_{X^{p,r}_{1,1}}^2.
\end{equation}

\subsubsection*{Step 3: Effective velocity}

To estimate the high frequencies of $\cQ u,$ we follow the approach of \cite{Haspot,Haspot2,Haspot3}, and  introduce
 the following ``effective'' velocity field\footnote{The idea is to write the term $\Delta\cQ u -\nabla a$ in
 \eqref{eq:lfbaro} as the Laplacian of some gradient-like vector-field.}:
$$
w:= \cQ u+(-\Delta)^{-1}\nabla a.
$$
We find out that
$$
\d_tw-\Delta w=-\cQ(u\cdot\nabla u)-\cQ(J(a)\wt\cA u)+k(a)\nabla a+\cQ(au)+w-(-\Delta)^{-1}\nabla a.
$$
Applying the heat estimates \eqref{eq:heat} for the high frequencies of $w$ only, we get
$$\displaylines{
\|w\|^h_{\wt L^\infty(\dot B^{d/p-1}_{p,1})\cap L^1(\dot B^{d/p+1}_{p,1})}\lesssim \|w_0\|^h_{\dot B^{d/p-1}_{p,1}}
+\|u\cdot\nabla u\|^h_{L^1(\dot B^{d/p-1}_{p,1})}+\|\cQ(J(a)\wt\cA u)\|^h_{L^1(\dot B^{d/p-1}_{p,1})}\hfill\cr\hfill+\|k(a)\nabla a\|^h_{L^1(\dot B^{d/p-1}_{p,1})}
+\|\cQ(au)\|^h_{L^1(\dot B^{d/p-1}_{p,1})}+\|w\|^h_{L^1(\dot B^{d/p-1}_{p,1})}+\|a\|^h_{L^1(\dot B^{d/p-2}_{p,1})}.}
$$
The important point is that, owing to the high frequency cut-off at $|\xi|\sim 2^{j_0},$
$$
\|w\|^h_{L^1(\dot B^{d/p-1}_{p,1})}\lesssim 2^{-2j_0}\|w\|^h_{L^1(\dot B^{d/p+1}_{p,1})}
\quad\hbox{and}\quad
\|a\|^h_{L^1(\dot B^{d/p-2}_{p,1})}\lesssim  2^{-2j_0}\|a\|^h_{L^1(\dot B^{d/p}_{p,1})}.
$$
Hence, if $j_0$ is large enough then the term $\|w\|^h_{L^1(\dot B^{d/p-1}_{p,1})}$ may be absorbed by the l.h.s.
The other terms satisfy quadratic estimates. Indeed, it is clearly the case of $u\cdot\nabla u$
according to \eqref{eq:Qu1}, for $\dot B^{d/2-1}_{2,1}$ embeds in $\dot B^{d/p-1}_{p,1}.$
Next, because the product maps $\dot B^{d/p}_{p,1}\times\dot B^{d/p-1}_{p,1}$ in $\dot B^{d/p-1}_{p,1},$ we have 
if $p<2d,$
$$
\|k(a)\nabla a\|_{L^1(\dot B^{d/p-1}_{p,1})}\lesssim\|a\|_{L^2(\dot B^{d/p}_{p,1})}^2.
$$
To handle $\cQ(J(a)\wt\cA u),$ we decompose it into
$$
\cQ(J(a)\wt\cA u)=T_{J(a)}\Delta\cQ u+\cQ R(J(a),\wt Au)+\cQ T_{\wt\cA u}J(a)+[\cQ,T_{J(a)}]\wt\cA u.
$$
Arguing as from proving \eqref{eq:Qu5}, we readily get 
$$
\|\cQ(J(a)\wt\cA u)\|_{L^1(\dot B^{d/p-1}_{p,1})}\lesssim\|a\|_{\wt L^\infty(\dot B^{d/p}_{p,1})}(\|\cQ u\|_{L^1(\dot B^{d/p+1}_{p,1})}+\| u\|_{\wt L^1(\dot B^{d/p+1}_{p,r})}).
$$
Finally, using Bony's decomposition, we see that
$$
\|au\|_{\wt L^1(\dot B^{d/p}_{p,r})}\lesssim\|a\|_{\wt L^2(\dot B^{d/p}_{p,1})}\|u\|_{\wt L^2(\dot B^{d/p}_{p,r}\cap \dot B^0_{\infty,1})}.
$$
Because 
$$\|\cQ(au)\|^h_{L^1(\dot B^{d/p-1}_{p,1})}\lesssim \|au\|_{\wt L^1(\dot B^{d/p}_{p,r})},
$$
we conclude that
\begin{equation}\label{eq:wh}
\|w\|^h_{\wt L^\infty(\dot B^{d/p-1}_{p,1})\cap L^1(\dot B^{d/p+1}_{p,1})}\lesssim \|w_0\|^h_{\dot B^{d/p-1}_{p,1}}
+2^{j_0}\|(a,u)\|_{X^{p,r}_{1,1}}^2+ 2^{-2j_0}\|a\|^h_{L^1(\dot B^{d/p}_{p,1})}.
\end{equation}

\subsubsection*{Step 4: High frequencies of the density}

We notice that
$$
\d_ta+u\cdot\nabla a+a=-a\,\div u-\div w.
$$
To bound the high frequencies of $a,$ we write that for all $j\geq j_0,$
$$
\d_t\ddj a+\dot S_{j-1}u\cdot\nabla\ddj a+\ddj a=-\ddj(T_{\nabla a}\cdot u+R(\nabla a,u)+a\,\div u+\div w)+R_j
$$
with $R_j:=\dot S_{j-1}u\cdot\nabla\ddj a-\ddj(T_u\cdot\nabla a).$
\medbreak
Arguing  as in \cite{D5}, we thus get for all $t\geq0,$
\begin{multline}\label{eq:ah0}
\|\ddj a(t)\|_{L^p}+\int_0^t\|\ddj a\|_{L^p}\,d\tau\leq\|\ddj a_0\|_{L^p}
+\frac1p\int_0^t\|\div\dot S_{j-1}u\|_{L^\infty}\|\ddj a\|_{L^p}\,d\tau
\\+\int_0^t\|\ddj(T_{\nabla a}\cdot u+R(\nabla a,u)+a\,\div u+\div w)\|_{L^p}\,d\tau
+\int_0^t\|R_j\|_{L^p}\,d\tau.
\end{multline}
Now, because $\dot B^{d/p}_{p,1}$ is an algebra, we may write
$$
\|a\,\div u\|_{\dot B^{d/p}_{p,1}}\lesssim \|a\|_{\dot B^{d/p}_{p,1}}\|\div u\|_{\dot B^{d/p}_{p,1}},
$$
and continuity results for the paraproduct, and remainder yield
$$\begin{array}{lll}
\|T_{\nabla a}\cdot u\|_{L^1(\dot B^{d/p}_{p,1})}&\lesssim& \|\nabla a\|_{\wt L^\infty(\dot B^{-1}_{\infty,1})}\|u\|_{\wt L^1(\dot B^{d/p+1}_{p,r})},\\[1ex]
\|R(\nabla a,u)\|_{L^1(\dot B^{d/p}_{p,1})}&\lesssim& \|\nabla a\|_{\wt L^\infty(\dot B^{d/p-1}_{p,1})}\|u\|_{\wt L^1(\dot B^{d/p+1}_{p,r})}.
\end{array}
$$
Finally,  because 
$$
R_j=\ddj\sum_{|j'-j|\leq 4}(\dot S_{j-1}-\dot S_{j'-1})u\cdot\nabla\dot\Delta_{j'}a+\sum_{|j'-j|\leq4}[\dot S_{j-1}u,\ddj]\cdot\nabla\dot\Delta_{j'}a,
$$
commutator estimates from \cite{BCD} lead to 
$$
\sum_{j\in\Z}2^{jd/p}\|R_j\|_{L^p}\leq C\|\nabla u\|_{L^\infty}\|a\|_{\dot B^{d/p}_{p,1}}.
$$
Multiplying \eqref{eq:ah0} by $2^{jd/p},$ using the above inequalities, and summing up over $j\geq j_0$ thus leads to
$$\displaylines{
\|a\|^h_{\wt L^\infty_t(\dot B^{d/p}_{p,1})}+\int_0^t\|a\|^h_{\dot B^{d/p}_{p,1}}\,d\tau \leq \|a_0\|^h_{\dot B^{d/p}_{p,1}}+ 
C\int_0^t\bigl(\|\nabla u\|_{L^\infty}+\|\div u\|_{\dot B^{d/p}_{p,1}}\bigr)\|a\|_{\dot B^{d/p}_{p,1}}\,d\tau\hfill\cr\hfill
+C \|\nabla a\|_{\wt L^\infty(\dot B^{d/p-1}_{p,1})}\|u\|_{\wt L^1(\dot B^{d/p+1}_{p,r})}
+C\|w\|^h_{L^1(\dot B^{d/p+1}_{p,1})}.}
$$
Therefore,
\begin{equation}\label{eq:ah}
\|a\|^h_{L^1(\dot B^{d/p}_{p,1})\cap\wt L^\infty(\dot B^{d/p}_{p,1})}\lesssim \|a_0\|^h_{\dot B^{d/p}_{p,1}}+2^{j_0}\|(a,u)\|_{X^{p,r}_{1,1}}^2+\|w\|^h_{L^1(\dot B^{d/p+1}_{p,1})}.
\end{equation}
Plugging \eqref{eq:wh} in \eqref{eq:ah} and taking $j_0$ large enough, we thus get
\begin{equation}\label{eq:ahbis}
\|a\|^h_{L^1(\dot B^{d/p}_{p,1})\cap\wt L^\infty(\dot B^{d/p}_{p,1})}\lesssim \|a_0\|^h_{\dot B^{d/p}_{p,1}}+\|\cQ u_0\|^h_{\dot B^{d/p-1}_{p,1}}+2^{j_0}\|(a,u)\|_{X_{1,1}^{p,r}}^2.
\end{equation}

\subsubsection*{Step 5: Closing the a priori estimates}

Resuming to \eqref{eq:wh} yields
\begin{equation}\label{eq:whbis}
\|w\|^h_{\wt L^\infty(\dot B^{d/p-1}_{p,1})\cap L^1(\dot B^{d/p+1}_{p,1})}\lesssim  \|a_0\|^h_{\dot B^{d/p}_{p,1}}+\|\cQ u_0\|^h_{\dot B^{d/p-1}_{p,1}}
+2^{j_0}\|(a,u)\|_{X^{p,r}_{1,1}}^2.
\end{equation}
As $\cQ u^h=w^h-(-\Delta)^{-1}\nabla a^h,$ the same inequality holds true for $\cQ u^h.$
Finally, putting together \eqref{eq:Pua}, \eqref{eq:Pub}, \eqref{eq:lf} and \eqref{eq:ahbis}, we conclude that
\begin{multline}\label{eq:end}
\|(a,u)\|_{X^{p,r}_{1,1}}\leq C\bigl(\|(a_0,\cQ u_0)\|^\ell_{\dot B^{d/2-1}_{2,1}}+\|\cP u_0\|_{\dot B^{d/p-1}_{p,r}\cap \dot B^{-1}_{\infty,1}}+\|a_0\|^h_{\dot B^{d/p}_{p,1}}\\+\|\cQ u_0\|^h_{\dot B^{d/p-1}_{p,1}}+2^{j_0}(1+\|(a,u)\|_{X^{p,r}_{1,1}})
\|(a,u)\|_{X^{p,r}_{1,1}}^2\bigr)\cdotp
\end{multline}
It is now easy to close the estimates if the data are small enough; we end up with \eqref{eq:ue}.

\subsubsection*{Step 6: The case of nonconstant viscosity coefficients}

It is only a matter of checking that the last line of \eqref{eq:NSC} satisfies quadratic estimates.
To this end, we write that
\begin{equation}\label{eq:add1}
\frac1{1+a}\div\bigl(\wt\mu(a) D(u)\bigr)=\frac{\wt\mu(a)}{1+a}\div D(u)
+\frac{\wt\mu'(a)}{1+a}D(u)\cdot\nabla a,
\end{equation}
and a similar relation for the term pertaining to $\wt\lambda.$
\medbreak
The first term of the r.h.s. of \eqref{eq:add1} may be handled exactly as $J(a)\wt\cA u.$ As for the second term, 
it suffices to estimate it in $L^1(\dot B^{d/p-1}_{p,1})$ and to show that applying $\cQ^\ell$ to it
leads to estimates in $L^1(\dot B^{d/2-1}_{2,1}).$ 
\smallbreak

Throughout, we use the fact  that 
$\frac{\wt\mu'(a)}{1+a}\nabla a=\nabla(L(a))$ for some smooth function $L$ vanishing at $0.$
Now, continuity properties of $R$ and $T$ imply that
$$
\begin{array}{rcl}
\|T_{\nabla(L(a))}\nabla u\|_{L^1(\dot B^{d/2-1}_{2,1})}&\lesssim& \|\nabla(L(a))\|_{\wt L^\infty(\dot B^{d/p^*-1}_{p^*,1})}
\|\nabla u\|_{\wt L^1(\dot B^{d/p}_{p,r})},\\[1ex]
\|R(\nabla(L(a)),\nabla u)\|_{L^1(\dot B^{d/2-1}_{2,1})}&\lesssim& \|\nabla(L(a))\|_{\wt L^\infty(\dot B^{d/p-1}_{p,1})}
\|\nabla u\|_{\wt L^1(\dot B^{d/p}_{p,r})},\\[1ex]
\|T_{\nabla u}\nabla(L(a))\|_{L^1(\dot B^{d/p-1}_{p,1})}&\lesssim& \|\nabla u\|_{L^\infty(L^\infty)}
\|\nabla(L(a))\|_{L^\infty(\dot B^{d/p-1}_{p,1})},
\end{array}
$$
which in particular yields quadratic estimates for the  $L^1(\dot B^{d/p-1}_{p,1})$ norm,
after using suitable embedding and the composition estimate \eqref{eq:Qu6}.

To complete the proof, it is only a matter of getting quadratic estimates for $\cQ^\ell(T_{\nabla u}\nabla L(a))$
in $L^1(\dot B^{d/2-1}_{2,1}).$ To this end, we observe that
$$
\|\cQ^\ell(T_{\nabla u}\nabla L(a))\|_{L^1(\dot B^{d/2-1}_{2,1})}\lesssim 
2^{j_0}\|\cQ^\ell(T_{\nabla u}\nabla L(a))\|_{L^1(\dot B^{d/2-2}_{2,1})}, $$
and thus 
$$\begin{array}{lll}
\|\cQ^\ell(T_{\nabla u}\nabla L(a))\|_{L^1(\dot B^{d/2-1}_{2,1})}&\!\!\!\lesssim\!\!\!& 
2^{j_0}\|\nabla u\|_{\wt L^2(\dot B^{d/p^*-1}_{p^*,r})}\|\nabla(L(a))\|_{\wt L^2(\dot B^{d/p-1}_{p,1})}\\[1ex]
&\!\!\!\lesssim\!\!\!&2^{j_0}\|u\|_{\wt L^2(\dot B^{d/p^*}_{p^*,r})}\|a\|_{\wt L^2(\dot B^{d/p}_{p,1})}.
\end{array}$$
Therefore we end up with 
$$
\biggl\|\frac{\wt\mu'(a)}{1+a}D(u)\cdot\nabla a\biggr\|_{L^1(\dot B^{d/p-1}_{p,1})}
\lesssim 2^{j_0}(1+\|a\|_{L^\infty(\dot B^{d/p}_{p,1})})\|(a,u)\|_{X^{p,r}_{1,1}}^2,
$$
and one may conclude that \eqref{eq:end} is still fulfilled in this more general situation.


\subsection{Existence of a global solution to System \eqref{eq:NSC}}
Let us now give a few words on the existence issue. 
The simplest way is to smooth out the initial velocity $u_0$ into
a sequence of initial velocities $(u_0^n)_{n\in\N}$  with $(u_0^n)^\ell$ in $\dot B^{d/2-1}_{2,1}$
uniformly, and $(u_0^n)^h$ in $\dot B^{d/p-1}_{p,1}.$
Then using the results of \cite{D0,D6} yields a unique local-in-time solution $(a^n,u^n)$ to 
\eqref{eq:NSC} with data $(a_0,u_0^n).$ From the above estimates, 
we know in addition that (with obvious notation)
$$
\|(a^n,u^n)\|_{X^{p,r}_{1,1}(0,t)}\leq C\bigl(\|(a_0^n,\cQ u_0^n)\|^\ell_{\dot B^{d/2-1}_{2,1}}+\|\cP u_0^n\|_{\dot B^{d/p-1}_{p,r}\cap \dot B^{-1}_{\infty,1}}+
\|a_0^n\|^h_{\dot B^{d/p}_{p,1}}+\|\cQ u_0^n\|^h_{\dot B^{d/p-1}_{p,1}}\bigr)
$$
is fulfilled whenever $t$ is smaller than the lifespan $T^n_*$ of $(a^n,u^n).$
As the above inequality implies that 
$$
\|a^n\|_{\wt L^\infty_{T_*^n}(\dot B^{d/p}_{p,1})}+\int_0^{T_*^n}\|\nabla u^n\|_{L^\infty}\,dt<\infty,
$$
a straightforward adaptation of Prop. 10.10 of \cite{BCD} to $p\not=2$ implies
that $T_*^n=+\infty.$ {} We thus have  for all $n\in\N,$ 
$$
\|(a^n,u^n)\|_{X^{p,r}_{1,1}}\leq C\bigl(\|(a_0,\cQ u_0)\|^\ell_{\dot B^{d/2-1}_{2,1}}+\|\cP u_0\|_{\dot B^{d/p-1}_{p,r}\cap \dot B^{-1}_{\infty,1}}+
\|a_0\|^h_{\dot B^{d/p}_{p,1}}+\|\cQ u_0\|^h_{\dot B^{d/p-1}_{p,1}}\bigr).
$$
Next,  compactness arguments similar to those of e.g. \cite{BCD} or \cite{D0} allow to conclude
that $(a^n,u^n)_{n\in\N}$ weakly converges (up to extraction) to some global solution of \eqref{eq:NSC} with 
the desired regularity properties, and satisfying \eqref{eq:ue} (with $\eps=\nu=1$ of course). 
Resuming to the original unknowns completes the proof of the first part of Theorem \ref{th:main1}. 



\section{The incompressible limit: weak convergence}

Granted with the uniform estimates established in the previous section,
it is now possible to pass to the limit in the system in the sense of distributions.
As in  the work by P.-L. Lions and N. Masmoudi \cite{LM}
dedicated to the finite energy weak solutions of \eqref{eq:NS}, the proof relies on compactness arguments, 
and works  the same in the $\R^d$ and $\T^d$ cases.
To simplify the presentation, we assume that 
the viscosity functions $\lambda$ and $\mu$ are constant. 
\medbreak
So we consider a family $(a_0^\eps,u_0^\eps)$ of data satisfying \eqref{eq:smalldata1}
and $\cP u^\eps_0\rightharpoonup v_0$ when $\eps$ goes to $0.$
We denote by $(a^\eps,u^\eps)$ the corresponding  solution of \eqref{eq:NS}
given by Theorem \ref{th:main1}.
Because
\begin{equation}\label{eq:conv0}
\|a_0^{\eps}\|_{\dot B^{d/p-1}_{p,1}}^{h,\wt\eps}\lesssim \wt\eps\|a_0^{\eps}\|_{\dot B^{d/p}_{p,1}}^{h,\wt\eps},
\end{equation}
the data  $(a^{\eps}_0,u_{0}^{\eps})$ are uniformly bounded
in $\dot B^{d/p-1}_{p,1}\times (\dot B^{d/p}_{p,r}\cap\dot B^{-1}_{\infty,1}),$ 
and thus in $\dot B^{d/4-1}_{4,2}.$ 
Likewise, \eqref{eq:ue} ensures that 
  $(a^{\eps},u^{\eps})$ is bounded in the space  $\wt\cC_b(\R_+;\dot B^{d/4-1}_{4,2}),$
  given our assumptions on $p$ and $r.$
Therefore there exists a sequence $(\eps_n)_{n\in\N}$ decaying to $0$ so that
$$
(a^{\eps_n}_0,u_0^{\eps_n})\rightharpoonup (a_0,u_0)
\quad\!\!\hbox{in}\!\!\quad \dot B^{d/4-1}_{4,2}\quad\hbox{and}\quad
(a^{\eps_n},u^{\eps_n})\rightharpoonup (a,u)
\quad\!\!\hbox{in}\!\!\quad L^\infty(\R_+;\dot B^{d/4-1}_{4,2})\quad\hbox{weakly}\ *.
$$
Of course, we have $\cP u_0= v_0.$
\medbreak
The strong convergence of the density to $1$  is obvious: we have
$\rho^{\eps_n}=1+\eps_na^{\eps_n},$
and  $(a^{\eps_n})_{n\in\N}$ is bounded (in $L^2(\R_+;\dot B^{d/p}_{p,1})$ for instance).
\medbreak
In order to justify  that $\div u=0,$ we rewrite the  mass equation as follows:
$$
\div u^{\eps_n}=-\eps_n\div(a^{\eps_n} u^{\eps_n})-\eps_n\d_ta^{\eps_n}.
$$
Given that $a^{\eps_n}$ and $u^{\eps_n}$ are bounded in $L^2(\R_+;\dot B^{d/4}_{4,2}\cap L^\infty)$
(use the definition of $X^{p,r}_{\eps,\nu}$ and interpolation), 
the first term in the right-hand side is  $\cO(\eps_n)$ in $L^1(\R_+;\dot B^{d/4-1}_{4,2}).$
As for the last term, it tends to $0$ in the sense of distributions, for $a^{\eps_n}\rightharpoonup a$ 
in $L^\infty(\R_+;\dot B^{d/4-1}_{4,2})$ weakly $*.$
We thus have $\div u^{\eps_n}\rightharpoonup 0,$ whence $\div u=0.$
\medbreak
To complete the proof of the weak convergence, 
it is only a matter of justifying that $u^{\eps_n}$  converges
in the sense of distributions to the solution $u$ of \eqref{eq:ins}.
To achieve it, we project the velocity equation onto divergence-free vector fields, and get
\begin{equation}\label{eq:conv1}
\d_t\cP u^{\eps_n}-\mu\Delta\cP u^{\eps_n}=-\cP( u^{\eps_n}\cdot\nabla u^{\eps_n})
-\cP\bigl(J(\eps_na^{\eps_n})\cA u^{\eps_n}\bigr).
\end{equation}
Because $\cQ u=0,$ the left-hand side weakly converges to $\d_t u-\mu\Delta u.$
\smallbreak
To prove that the last term tends to $0,$ we use the fact that having $\wt\eps (a^{\eps})^{h,\wt\eps}$
and $(a^{\eps})^{\ell,\wt\eps}$ bounded in $\wt L^\infty(\dot B^{d/p}_{p,1})$ and $\wt L^\infty(\dot B^{d/p-1}_{p,1}),$ respectively, implies that, for all $\alpha\in[0,1],$
\begin{equation}\label{eq:aa}
\wt\eps^\alpha a^{\eps}\quad\hbox{is bounded in }\ \wt L^\infty(\dot B^{d/p-1+\alpha}_{p,1}).
\end{equation}
Now  $\cA u^\eps$ is bounded in $\wt L^1(\dot B^{d/p-1}_{p,r})$ and   $p<2d.$ 
Hence,  according to product laws in Besov spaces, composition inequality and \eqref{eq:aa}, we get
   $J(\eps a^{\eps})\cA u^{\eps}=\cO(\wt\eps^{1-\alpha})$
in $\wt L^1(\dot B^{d/p-2+\alpha}_{p,r}),$ whenever $2\max(0,1-d/p)<\alpha\leq1.$
\smallbreak
In order to  prove that $\cP(u^{\eps_n}\cdot\nabla u^{\eps_n})\rightharpoonup \cP(u\cdot\nabla u),$ we  note that
$$u^{\eps_n}\cdot\nabla u^{\eps_n}=\frac12\nabla|\cQ u^{\eps_n}|^2+\cP u^{\eps_n}\cdot\nabla u^{\eps_n}
+\cQ u^{\eps_n}\cdot\nabla\cP u^{\eps_n}.$$
Projecting the first term onto divergence free vector fields gives $0,$ and we also know that $\cP u=u.$
Hence we just have to prove that
\begin{equation}\label{eq:conv2}
\cP(\cP u^{\eps_n}\cdot\nabla u^{\eps_n})\rightharpoonup \cP(\cP u\cdot\nabla u)
\quad\hbox{and}\quad
\cP(\cQ u^{\eps_n}\cdot\nabla\cP u^{\eps_n})\rightharpoonup0.
\end{equation}
This requires our proving results of strong convergence for $\cP u^{\eps_n}.$ 
To this end, we shall  exhibit uniform  bounds  for $\d_t\cP u^{\eps_n}$ in a suitable space. 
First, arguing by interpolation, we see that $(\nabla^2 u^{\eps_n})$ is bounded
in $\wt L^m(\dot B^{d/p+2/m-3}_{p,r})$ for any $m\geq1.$ 
Choosing $m>1$ so that $2d/p+2/m-3>0$ (this is possible as $p<2d$) and remembering
that $(\eps^na^{\eps_n})$ is bounded in $\wt L^\infty(\dot B^{d/p}_{p,1}),$ we thus 
get $(J(\eps_na^{\eps_n})\cA u^{\eps_n})$ bounded in $\wt L^m(\dot B^{d/p+2/m-3}_{p,r}).$
Similarly, combining the facts that $(u^{\eps_n})$ and $(\nabla u^{\eps_n})$ are bounded in 
 $\wt L^\infty(\dot B^{d/p-1}_{p,r})$ and $\wt L^m(\dot B^{d/p+2/m-2}_{p,r}),$ respectively, we see that 
$(u^{\eps_n}\cdot\nabla u^{\eps_n})$ is  bounded in $\wt L^m(\dot B^{d/p+2/m-3}_{p,r}),$ too. 
Computing $\d_t\cP u^{\eps_n}$ from \eqref{eq:conv1}, it is now clear that 
 $(\d_t\cP u^{\eps_n})$ is bounded in  $\wt L^{m}(\dot B^{d/p+2/m-3}_{p,r}).$
 Hence  $(\cP u^{\eps_n}-\cP u^{\eps_n}_0)$ is bounded in $\cC^{1-1/m}(\R_+; \dot B^{d/p+2/m-3}_{p,r}).$
As $\cP u^{\eps_n}$ is also bounded in $\wt\cC_b(\R_+;\dot B^{d/p-1}_{p,r}),$
and as the embedding of $\dot B^{d/p-1}_{p,1}$ in $\dot B^{d/p+2/m-3}_{p,1}$ is locally compact
(see e.g. \cite{BCD}, page 108), we conclude by means of Ascoli theorem that, up to a new extraction,  for all $\phi\in\cS(\R^d)$ and $T>0,$
\begin{equation}\label{eq:conv3}
\phi\cP u^{\eps_n}\longrightarrow \phi\cP u \quad\hbox{in}\quad  \cC([0,T];\dot B^{d/p+2/m-3}_{p,1}).
\end{equation}
Interpolating  with the uniform in $\cC_b(\R_+;\dot B^{d/p-1}_{p,r}),$ we can upgrade
the strong convergence in \eqref{eq:conv3}
 to the space $\cC([0,T];\dot B^{d/p-1-\alpha}_{p,1})$ for all small enough $\alpha>0,$
and all $T>0.$ Combining with the properties of weak convergence
for $\nabla u^{\eps_n}$ to $\nabla u,$ and $\cQ u^{\eps_n}$ to $0$ that may be deduced
from the uniform bounds on $u^{\eps_n},$ it is now easy to conclude to \eqref{eq:conv2}.


\section{The incompressible limit: strong convergence in the whole space case}\label{s:strong}

In this section, we combine  Strichartz estimates for the following acoustic wave equations 
\begin{equation}\label{eq:acoustic}
\left\{\begin{array}{l}
\d_ta^\eps+\Frac{\div u^\eps}\eps=F^\eps,\\[1ex]
\d_tu^\eps+\Frac{\nabla a^\eps}\eps=G^\eps
\end{array}\right.\qquad (t,x)\in\R_+\times\R^d
\end{equation}
associated to \eqref{eq:NS}, with the uniform bounds \eqref{eq:ue} for the constructed solution
$(a^\eps,u^\eps)$ so as to establish the strong convergence for $u^\eps$ to the solution 
$v$ of \eqref{eq:ins} in a proper function space. 
Recall that in a different context (that of global weak solutions), the idea of taking advantage of 
Strichartz estimates for investigating the incompressible limit goes back to the work 
of B. Desjardins and E. Grenier in \cite{DG}.

Throughout the proof, we assume the viscosity coefficients to be constant, for simplicity.
Recall that $C_0^{\eps,\nu}$ denotes the l.h.s. of \eqref{eq:smalldata1}.

 We first  consider the case $d\geq3$ which is slightly easier than the
two-dimensional case, owing to more available Strichartz estimates.

\subsubsection*{The case $d\geq3$}

Let us assume that $\eps=\nu=1$ for a while. Then the solution $(a,u)$ to \eqref{eq:NS} satisfies \eqref{eq:acoustic}
 with $F=-\div(au)$ and $G=\Delta\cQ u-\cQ(u\cdot\nabla u)-\cQ(J(a)\wt\cA u)+k(a)\nabla a,$
and  Proposition 2.2 in \cite{D3} ensures that  for all $q\in[2,\infty),$ we have 
$$ \|(a,\cQ u)\|^\ell_{\wt L^{2q/(q-2)}(\dot B^{(d-1)/q-1/2}_{q,1})}\lesssim 
\|(a_0,\cQ u_0)\|^\ell_{\dot B^{d/2-1}_{2,1}}+\|(F,G)\|_{L^1(\dot B^{d/2-1}_{2,1})}^\ell.
 $$
 {}Following  the proof of \eqref{eq:lf} to bound $F$ and $G,$ we eventually get 
 $$
  \|(a,\cQ u)\|^\ell_{\wt L^{2q/(q-2)}(\dot B^{(d-1)/q-1/2}_{q,1})}\lesssim C_0^{1,1}.
  $$
  As  we also have 
  $$
   \|(a,\cQ u)\|^\ell_{L^{1}(\dot B^{d/2+1}_{2,1})}\lesssim C_0^{1,1},
 $$
 we conclude by using the following complex interpolation result
 $$
 [L^1(\dot B^{d/2+1}_{2,1}),\wt L^{2q/(q-2)}(\dot B^{(d-1)/q-1/2}_{q,1})]_{q/(q+2)}=\wt L^2(\dot B^{(d+1)/p-1/2}_{p,1})
 \quad\hbox{with }\ p=(q+2)/2,
 $$
 that 
 \beno \|(a,\cQ u)\|^\ell_{\wt L^{2}(\dot B^{(d+1)/p-1/2}_{p,1})}\lesssim C_0^{1,1}\quad\hbox{for all }\ p\in[2,+\infty).   \eeno
Back to the original variables in  \eqref{eq:change}, we deduce that for all positive $\eps$ and $\nu,$  
\beno \nu^{1/2}\|(a^\eps,\cQ u^\eps)\|^{\ell,\wt\eps}_{\wt L^{2}(\dot B^{(d+1)/p-1/2}_{p,1})}\lesssim 
\wt\eps^{1/2-1/p}\,C_0^{\eps,\nu}.   \eeno
Of course, for the above inequality to be true, we need in addition that  the index $p$ fulfills the assumptions in Theorem \ref{th:main1}. 
Now, taking advantage of  the high-frequency cut-off (second line below) and 
\eqref{eq:ue} (third line),  we get  
 \beno\|(a^\eps,\cQ u^\eps)\|_{\wt L^2(\dot B^{(d+1)/p-1/2}_{p,1})}&\!\!\!\!\lesssim\!\!\!\! & \|(a^\eps,\cQ u^\eps)\|^{\ell,\wt\eps}_{\wt L^2(\dot B^{(d+1)/p-1/2}_{p,1})}+\|(a^\eps,\cQ u^\eps)\|^{h,\wt\eps}_{\wt L^2(\dot B^{(d+1)/p-1/2}_{p,1})} \\
&\!\!\!\!\lesssim\!\!\!\! & \| (a^\eps,\cQ u^\eps)\|^{\ell,\wt\eps}_{\wt L^2(\dot B^{(d+1)/p-1/2}_{p,1})}\!+\!
\wt\eps^{\,1/2-1/p}\|(a^\eps,\cQ u^\eps)\|^{h,\wt\eps}_{\wt L^2(\dot B^{d/p}_{p,1})}\\
&\!\!\!\!\lesssim\!\!\!\! &\nu^{-1/2}\,\wt\eps^{\,1/2-1/p}\,C_0^{\eps,\nu},
\eeno
which yields the strong convergence of $(a^\eps,\cQ u^\eps)$ to $0$ in $\wt L^2(\dot B^{(d+1)/p-1/2}_{p,1}),$ with an explicit rate.
\medbreak
Let us now go to the proof of the convergence of $\cP u^\eps.$
Setting $\du^\eps:=\cP u^\eps-u,$ we see that 
$$
\d_t \du^\eps-\mu \Delta\du^\eps+
\cP (\cP u^\eps\cdot \na \du^\eps+\du^\eps\cdot\na u) =-\cP\Bigl(u^\eps\cdot \na \cQ u^\eps + \cQ u^\eps\cdot\na \cP u^\eps+
J(\eps a^\eps)\mathcal{A} u^\eps\Bigr)\cdotp
$$
In what follows, we aim at estimating $\du^\eps$ in the  space $\wt L^\infty(\dot B^{(d+1)/p-3/2}_{p,r})\cap \wt L^1(\dot B^{(d+1)/p+1/2}_{p,r}).$  
First, applying \eqref{eq:heat} and the fact that $\cP$ is a self-map in any homogeneous Besov space gives
$$
\displaylines{\delta\!U^\eps:=\|\du^\eps\|_{\wt L^\infty(\dot B^{(d+1)/p-3/2}_{p,r})}+\mu\|\du^\eps\|_{\wt L^1(\dot B^{(d+1)/p+1/2}_{p,r})}
\hfill\cr\hfill\lesssim
\|\du_0^\eps\|_{\dot B^{(d+1)/p-3/2}_{p,r}}+ \|\cP u^\eps\cdot \na \du^\eps+\du^\eps\cdot\na u\|_{\wt L^1(\dot B^{(d+1)/p-3/2}_{p,r})} \hfill\cr\hfill+
\|u^\eps\cdot \na \cQ u^\eps + \cQ u^\eps\cdot\na \cP u^\eps+J(\eps a^\eps)\mathcal{A} u^\eps\|_{\wt L^1(\dot B^{(d+1)/p-3/2}_{p,r})}.}
$$

Next, product and composition estimates in the spirit of those of the previous sections (where we use repeatedly  that
  $(d+1)/p-1/2\leq d/p$ and  $(d+1)/p-3/2+d/p>0$) yield:
  \beno
\|\cP u^\eps\cdot \na \du^\eps\|_{\wt L^1(\dot B^{(d+1)/p-3/2}_{p,r})}&\lesssim&\| \cP u^\eps\|_{\wt L^\infty(\dot B^{d/p-1}_{p,r})}
\|\nabla\du^\eps\|_{\wt L^1(\dot B^{(d+1)/p-1/2 }_{p,r})}\\
&&\qquad\qquad+\| \cP u^\eps\|_{\wt L^1(\dot B^{d/p+1}_{p,r})}\|\nabla\du^\eps\|_{\wt L^\infty(\dot B^{(d+1)/p-5/2 }_{p,r})},
\\ \|\du^\eps\cdot\na u\|_{\wt L^1(\dot B^{(d+1)/p-3/2}_{p,r})}
&\lesssim&
\| \nabla u\|_{\wt L^\infty(\dot B^{d/p-2}_{p,r})}\|\du^\eps\|_{\wt L^1(\dot B^{(d+1)/p+1/2 }_{p,r})}\\&&\qquad\qquad+
\|\nabla u\|_{\wt L^1(\dot B^{d/p}_{p,r})}\|\du^\eps\|_{\wt L^\infty(\dot B^{(d+1)/p-3/2 }_{p,r})},\eeno
and also  
\beno \|u^\eps\cdot \na \cQ u^\eps\|_{\wt L^1(\dot B^{(d+1)/p-3/2}_{p,r})}&\lesssim&
\|\na \cQ u^\eps\|_{\wt L^2(\dot B^{(d+1)/p-3/2}_{p,1})}\|u^\eps\|_{\wt L^2(\dot B^{d/p }_{p,r}\cap \dot B^0_{\infty,1})},\\
 \|\cQ u^\eps\cdot \na \cP u^\eps\|_{\wt L^1(\dot B^{(d+1)/p-3/2}_{p,r})}&\lesssim&
\| \cQ u^\eps\|_{\wt L^2(\dot B^{(d+1)/p-1/2}_{p,1})}\|\nabla\cP u^\eps\|_{\wt L^2(\dot B^{d/p-1 }_{p,r})},\\
\|J(\eps a^\eps)\mathcal{A} u^\eps\|_{\wt L^1(\dot B^{(d+1)/p-3/2}_{p,r})}&\lesssim&
\| J(\eps a^\eps)\|_{\wt L^\infty(\dot B^{(d+1)/p-1/2}_{p,1})}\|\mathcal{A} u^\eps\|_{\wt L^1(\dot B^{d/p-1}_{p,r})}\\
&\lesssim& (1+\|\eps a^\eps\|_{\wt L^\infty(\dot B^{d/p}_{p,1})})\|\eps a^\eps\|_{\wt L^\infty(\dot B^{(d+1)/p-1/2}_{p,1})} \|u^\eps\|_{L^1(\dot B^{d/p+1}_{p,r})}.
 \eeno
Let us observe that
 \begin{eqnarray}
\quad \| \eps a^\eps\|_{\wt L^\infty(\dot B^{(d+1)/p-1/2}_{p,1})}&\!\!\!\lesssim\!\!\! & \| \eps a^\eps\|^{\ell,\wt\eps}_{\wt L^\infty(\dot B^{(d+1)/p-1/2}_{p,1})}+\| \eps a^\eps\|^{h,\wt\eps}_{\wt L^\infty(\dot B^{(d+1)/p-1/2}_{p,1})}\nonumber \\
&\!\!\!\lesssim\!\!\!& \nu^{-1}\wt\eps^{\,1/2-1/p}\|a^\eps\|^{\ell,\wt\eps}_{\wt L^\infty(\dot B^{d/2-1}_{2,1})} +\wt\eps^{1/2-1/p}\| \eps a^\eps\|^{h,\wt\eps}_{\wt L^\infty(\dot B^{d/p}_{p,1})}\nonumber\\&\!\!\!\lesssim\!\!\!& \nu^{-1}\wt\eps^{\,1/2-1/p} C_0^{\eps,\nu}.\label{eq:conv4}
\end{eqnarray}
Therefore, putting  together all the above estimates and using \eqref{eq:ue}, we get
$$
\displaylines{
\delta\!U^\eps \lesssim
\|\du_0^\eps\|_{\dot B^{(d+1)/p-3/2}_{p,r}}\hfill\cr\hfill+ 
\mu^{-1}\bigl( \|u\|_{\wt L^\infty(\dot B^{d/p-1}_{p,r})}+\mu\|u\|_{\wt L^1(\dot B^{d/p+1}_{p,r})}\bigr)\delta\! U^\eps
+ \nu^{-1}\wt\eps^{\,1/2-1/p} (1+\nu^{-1}C_0^{\eps,\nu})(C_0^{\eps,\nu})^2.}
$$
Note that Theorem \ref{thm:NS} implies that as $v_0$ is small compared to $\mu$  
(a consequence of smallness condition \eqref{eq:smalldata1}) then the solution $u$ to \eqref{eq:ins} 
with data $v_0$ exists globally and satisfies \eqref{eq:estins}. 
We thus get 
$$
\delta\!U^\eps \lesssim\|\du_0^\eps\|_{\dot B^{(d+1)/p-3/2}_{p,r}}+\wt\eps^{\,1/2-1/p} C_0^{\eps,\nu},
$$
which completes the proof of convergence in $\R^d$ if  $d\geq3.$


\subsubsection*{The case $d=2$}

Applying  Proposition 2.2 in \cite{D3} to \eqref{eq:acoustic} in the case $d=2,$ and using the bounds
of the previous section to bound the r.h.s. in $L^1(\dot B^0_{2,1}),$  we now get if $\eps=\nu=1,$
\beno \|(a,\cQ u)\|^\ell_{\wt L^{r}(\dot B^{2/q-1+1/r }_{q,1})}\lesssim C_0^{1,1}  \quad\hbox{whenever }\  2/r\leq 1/2-1/q.  \eeno
Let us emphasize that in contrast with the high-dimensional case, we cannot have $r$ smaller than $4.$
In what follows, we set $1/r=c(1/2-1/q)$ with $c\in[0,1/2]$ to be fixed later on. Observing that \eqref{eq:ue} implies that
$$
\|(a,\cQ u)\|^\ell_{L^1(\dot B^{2}_{2,1})}\lesssim C_0^{1,1}
$$
and adapting the interpolation argument used in the previous paragraph, 
we get 
\beno \|(a,\cQ u)\|^\ell_{\wt L^{2}(\dot B^{(c+2)/p-c/2}_{p,1})}\lesssim C_0^{1,1},  \eeno
where $p,$ $c$ and $q$ are interrelated through
$$
p=\frac{4q+(4-2q)c}{q+2+(2-q)c}\cdotp
$$
Note that as $c\in[0,1/2]$ and $q\in[2,+\infty],$ one can achieve any $p\in[2,6],$
which is a weaker condition than that which is imposed for $p$ in the statement of Theorem \ref{th:main1}. 
\medbreak
For general $\eps$ and $\nu,$ the above inequality recasts in 
\beno\nu^{1/2} \|(a^\eps,\cQ u^\eps)\|^{\ell,\wt\eps}_{\wt L^2(\dot B^{(c+2)/p-c/2}_{p,1})}\lesssim \wt\eps^{\,c(1/2-1/p)} C_0^{\eps,\nu}. \eeno
Arguing as in the high-dimensional case, one can get a similar inequality for the high frequencies of $(a^\eps,\cQ u^\eps),$ namely
  \beno\| (a^\eps,\cQ u^\eps)\|_{\wt L^2(\dot B^{(c+2)/p-c/2}_{p,1})}&\!\!\!\lesssim\!\!\! &
   \|(a^\eps,\cQ u^\eps)\|^{\ell,\wt\eps}_{\wt L^2(\dot B^{(c+2)/p-c/2}_{p,1})}+\| (a^\eps,\cQ u^\eps)\|^{h,\wt\eps}_{\wt L^2(\dot B^{(c+2)/p-c/2}_{p,1})} \\
&\!\!\!\lesssim\!\!\!&  \|(a^\eps,\cQ u^\eps)\|^{\ell,\wt\eps}_{\wt L^2(\dot B^{(c+2)/p-c/2}_{p,1})}
+\wt\eps^{\,c(1/2-1/p)}\|(a^\eps,\cQ u^\eps)\|^{h,\wt\eps}_{\wt L^2(\dot B^{2/p}_{p,1})}\\
&\!\!\!\lesssim\!\!\!& \nu^{-1/2}\wt\eps^{\,c(1/2-1/p)}C_0^{\eps,\nu}.
\eeno
Let us finally prove the convergence of $\cP u^\eps$ to $u$ in 
 $\wt L^\infty(\dot B^{(c+2)/p-c/2-1}_{p,r})\cap \wt L^1(\dot B^{(c+2)/p-c/2+1}_{p,r}).$
 Again, we apply Inequality \eqref{eq:heat} to the equation fulfilled by $\du^\eps,$ and get
 $$\displaylines{
 \delta\! U^\eps:= \|\du^\eps\|_{\wt L^\infty(\dot B^{(c+2)/p-c/2-1}_{p,r})}+\|\du^\eps\|_{\wt L^1(\dot B^{(c+2)/p-c/2+1}_{p,r})}\lesssim
 \|\du^\eps_0\|_{\dot B^{(c+2)/p-c/2-1}_{p,r}}
 \hfill\cr\hfill+ \|\cP u^\eps\cdot \na \du^\eps+\du^\eps\cdot\na u\|_{\wt L^1(\dot B^{(c+2)/p-c/2-1}_{p,r})} \hfill\cr\hfill+
\|u^\eps\cdot \na \cQ u^\eps + \cQ u^\eps\cdot\na \cP u^\eps+J(\eps a^\eps)\mathcal{A} u^\eps\|_{\wt L^1(\dot B^{(c+2)/p-c/2-1}_{p,r})}.}
 $$
In order to bound the nonlinear terms, we use standard continuity results for the product or paraproduct, 
and also (repeatedly) the fact that the condition on $c$ in Theorem \ref{th:main1} is equivalent to  $(c+2)/p-c/2-1+2/p>0$. 
Then we get 
\beno
\|\cP u^\eps\cdot \na \du^\eps\|_{\wt L^1(\dot B^{(c+2)/p-c/2-1}_{p,r})}&\lesssim&
\| \cP u^\eps\|_{\wt L^\infty(\dot B^{2/p-1}_{p,r})}\|\nabla\du^\eps\|_{\wt L^1(\dot B^{(c+2)/p-c/2 }_{p,r})}\\&&
\qquad\qquad+
\| \cP u^\eps\|_{\wt L^1(\dot B^{2/p+1}_{p,r})}\|\nabla\du^\eps\|_{L^\infty(\dot B^{(c+2)/p-c/2-2}_{p,r})},\\
 \|\du^\eps\cdot\na u\|_{\wt L^1(\dot B^{(c+2)/p-c/2-1}_{p,r})}&\lesssim&
\| \na u\|_{\wt L^\infty(\dot B^{2/p-2}_{p,r})}\|\du^\eps\|_{\wt L^1(\dot B^{(c+2)/p-c/2+1 }_{p,r})}\\&&
\qquad\qquad+
\|\nabla u\|_{\wt L^1(\dot B^{2/p}_{p,r})}\|\du^\eps\|_{\wt L^\infty(\dot B^{(c+2)/p-c/2-1 }_{p,r})},\eeno
and also 
\beno \|u^\eps\cdot \na \cQ u^\eps\|_{\wt L^1(\dot B^{(c+2)/p-c/2-1}_{p,r})}&\!\!\!\lesssim\!\!\!&
\| \na\cQ u^\eps\|_{\wt L^2(\dot B^{(c+2)/p-c/2-1}_{p,1})}\|u^\eps\|_{\wt L^2(\dot B^{2/p }_{p,r}\cap\dot B^0_{\infty,1})},\\
 \|\cQ u^\eps\cdot \na\cP u^\eps\|_{\wt L^1(\dot B^{(c+2)/p-c/2-1}_{p,r})}&\!\!\!\lesssim\!\!\!&
\| \cQ u^\eps\|_{\wt L^2(\dot B^{(c+2)/p-c/2}_{p,1})}\|\na\cP u^\eps\|_{\wt L^2(\dot B^{2/p-1}_{p,r})},\\
\|J(\eps a^\eps)\mathcal{A} u^\eps\|_{\wt L^1(\dot B^{(c+2)/p-c/2-1}_{p,r})}
&\!\!\!\lesssim\!\!\!& (1+\|\eps a^\eps\|_{\wt L^\infty(\dot B^{2/p}_{p,1})})\|\eps a^\eps\|_{\wt L^\infty(\dot B^{(c+2)/p-c/2}_{p,1})}
 \|u^\eps\|_{\wt L^1(\dot B^{2/p+1}_{p,r})}.
 \eeno
 In order to bound $\eps a^\eps$ in $\wt L^\infty(\dot B^{(c+2)/p-c/2}_{p,1}),$ one may argue exactly as in the case $d\geq3$:
\beno\| \eps a^\eps\|_{\wt L^\infty(\dot B^{(c+2)/p-c/2}_{p,1})}&\lesssim & \| \eps a^\eps\|^{\ell,\wt\eps}_{\wt L^\infty(\dot B^{(c+2)/p-c/2}_{p,1})}
+\|\eps a^\eps\|^{h,\wt\eps}_{\wt L^\infty(\dot B^{(c+2)/p-c/2}_{p,1})} \\
&\lesssim& \nu^{-1}\wt\eps^{\,c(1/2-1/p)}\|a^\eps\|^{\ell,\wt\eps}_{\wt L^\infty(\dot B^{0}_{2,1})} +\wt\eps^{\,c(1/2-1/p)}\| \eps a^\eps\|^{h,\wt\eps}_{\wt L^\infty(\dot B^{2/p}_{p,1})}\\&\lesssim& \nu^{-1}\wt\eps^{\,c(1/2-1/p)}C_0^{\eps,\nu}.
\eeno
So using Theorem \ref{thm:NS} to bound the terms pertaining to $u,$ it is now easy to conclude
to the last inequality of Theorem \ref{th:main1}. 
\ef


\section{The full Navier-Stokes-Fourier system}\label{s:full}

In this final section, we aim at extending the previous results to the more physically relevant case 
of non-isothermal polytropic fluids. 
The corresponding governing equations, the so-called  Navier-Stokes-Fourier system, 
 involves the density of the fluid  $\rho^\eps$ and its velocity $u^\eps.$
 To fully describe the fluid, we need to consider a \emph{third} (real valued) unknown, for instance 
 the temperature  $\theta^\eps.$ 

For simplicity, we only consider the case of perfect heat conducting and viscous gases.
We set the reference density and temperature to be $1,$ and focus on
 \emph{ill-prepared data} of the form
   $\rho_0^\eps=1+\eps a_0^\eps,$   $u_0^\eps$ and $\theta_0^\eps=1+\eps\vartheta_0^\eps$
   where $(a_0^\eps,u_0^\eps,\vartheta_0^\eps)$ are bounded in a sense that will be specified 
   later on\footnote{The reader may refer to \cite{FN-book} for the construction 
   and the low Mach asymptotic of the weak solutions to the  Navier-Stokes-Fourier equations, 
   and to \cite{Al} for the case of smoother data with large entropy variations.}.
 Setting $\rho^\eps =1+\eps a^\eps$ and $\theta^\eps=1+\eps\vartheta^\eps,$
we get the following system for $(a^\eps,u^\eps,\vartheta^\eps)$:
\begin{equation}
  \left\{
    \begin{aligned}
      & \d_t a^\eps+\frac{\div u^\eps}{\eps}=-\div (a^\eps u^\eps),\\
      &\d_t u^\eps + u^\eps\cdot\nabla u^\eps-\frac{\cA u^\eps}{1+\eps a^\eps}
      +\frac{\nabla(a^\eps+\vartheta^\eps+\eps a^\eps\vartheta^\eps)}{\eps(1+\eps a^\eps)}=0,\\
        &\d_t\vartheta^\eps+\frac{\div u^\eps}\eps+\div(\vartheta^\eps u^\eps)-\kappa\frac{\Delta\vartheta^\eps}{1+\eps a^\eps}
        =\frac\eps{1+\eps a^\eps}\Bigl(2\mu|Du^\eps|^2+\lambda(\div u^\eps)^2\Bigr)\cdotp
        \end{aligned}
  \right.
  \label{eq:NSF}
\end{equation}
We assume that the fluid is genuinely viscous and heat-conductive, that is to say
$$
\mu>0, \quad \nu:=\lambda+2\mu>0\ \hbox{ and }\ \kappa>0.
$$
Even though our results should hold for coefficients $\lambda,$ $\mu$ and $\kappa$ depending smoothly 
on the density, we only consider the constant case, for simplicity. 
\medbreak
Keeping in mind our results on the barotropic case, we want to consider families 
of small data $(a_0^\eps,u_0^\eps,\vartheta_0^\eps)$
in the space $Y^p_{0,\eps,\nu}$ defined by (still setting $\wt\eps:=\eps\nu$):
\begin{itemize}
\item $(a_0^\eps,\cQ u_0^\eps,\vartheta_0^\eps)^{\ell,\wt\eps}\in\dot B^{d/2-1}_{2,1},$
\item $(a_0^\eps)^{h,\wt\eps}\in \dot B^{d/p}_{p,1},\quad (\cQ u_0^\eps)^{h,\wt\eps}\in \dot B^{d/p-1}_{p,1}, \quad (\vartheta_0^\eps)^{h,\wt\eps}\in\dot B^{d/p-2}_{p,1},$
\item $\cP u_0^\eps\in \dot B^{d/p-1}_{p,1}.$
\end{itemize}
The existence space $Y^p_{\eps,\nu}$ is the set of triplets $(a,u,\vartheta)$ so that
\begin{itemize}
\item  $(a^{\ell,\wt\eps},\cQ u^{\ell,\wt\eps},\vartheta^{\ell,\wt\eps})\in \cC_b(\R_+;\dot B^{d/2-1}_{2,1})\cap L^1(\R_+;\dot B^{d/2+1}_{2,1}),$
\item  $a^{h,\wt\eps}\in \cC_b(\R_+;\dot B^{d/p}_{p,1})\cap L^1(\R_+; \dot B^{d/p}_{p,1}),$
\item $\vartheta^{h,\wt\eps}\in\cC_b(\R_+;\dot B^{d/p-2}_{p,1})\cap L^1(\R_+;\dot B^{d/p}_{p,1}),$
\item $\cQ u^{h,\wt\eps}$ and $\cP u^\eps$ are in $\cC_b(\R_+;\dot B^{d/p-1}_{p,1})\cap L^1(\R_+;\dot B^{d/p+1}_{p,1}),$
\end{itemize}
endowed with the norm:
$$\displaylines{
\|(a,u,\vartheta)\|_{Y^p_{\eps,\nu}}:= \|(a,\cQ u,\vartheta)\|^{\ell,\wt\eps}_{L^\infty(\dot B^{d/2-1}_{2,1})}+\|(\cP u,\cQ u^{h,\wt\eps})\|_{L^\infty(\dot B^{d/p-1}_{p,1})}
+\wt\eps\|a\|^{h,\wt\eps}_{L^\infty(\dot B^{d/p}_{p,1})}\hfill\cr\hfill+\wt\eps^{-1}\|\vartheta\|^{h,\wt\eps}_{L^\infty(\dot B^{d/p-2}_{p,1})}+\nu\|(a,\cQ u,\vartheta)\|^{\ell,\wt\eps}_{L^1(\dot B^{d/2+1}_{2,1})}+\nu\|(\cP u,\cQ u^{h,\wt\eps})\|_{L^1(\dot B^{d/p+1}_{p,1})}
+\eps^{-1}\|(a,\vartheta)\|^{h,\wt\eps}_{L^1(\dot B^{d/p}_{p,1})}.}
$$
We also set
$$\displaylines{\quad
\|(a_0,u_0,\vartheta_0)\|_{Y^p_{0,\eps,\nu}}:= \|(a_0,\cQ u_0,\vartheta_0)\|^{\ell,\wt\eps}_{\dot B^{d/2-1}_{2,1}}
\hfill\cr\hfill+\|(\cP u_0,\cQ u_0^{h,\wt\eps})\|_{\dot B^{d/p-1}_{p,1}}
+\wt\eps\|a_0\|^{h,\wt\eps}_{\dot B^{d/p}_{p,1}}+\wt\eps^{-1}\|\vartheta_0\|^{h,\wt\eps}_{\dot B^{d/p-2}_{p,1}}.\quad}
$$
Here the integer $j_0$ appearing in the threshold between low and high frequencies 
depends  only on $\wt\kappa:=\kappa/\nu,$ $\wt\mu:=\mu/\nu$ and $\wt\lambda:=\lambda/\nu$
 with $\nu:=\lambda+2\mu.$ 
\medbreak
In the case $p=2$ and $\eps=1,$  global existence for \eqref{eq:NSF} in the above space and for small data  has 
been established in \cite{D2}. 
The main goal of this section is to extend the statement to more general $p$'s, and to get estimates independent of $\eps$ and $\nu$
for the constructed solution. Furthermore, in the $\R^d$ case, we establish a strong convergence result
in the low Mach number asymptotics, in the spirit of our recent work \cite{DH}. 
 Here is the main  result of this section:
\begin{thm}\label{th:main2} Assume that the fluid domain is either $\R^d$ or $\T^d$ with $d\geq3,$ and 
that the initial data $(a_0^\eps, u_0^\eps,\vartheta^\eps_0)$ are as above with $2\leq p<d$ and $p\leq 2d/(d-2).$
 There exists a constant $\eta$ independent of $\eps$ and of $\nu$ (but depending on $\kappa/\nu$) such that  if
\begin{equation}\label{eq:smalldataf}
\|(a_0^\eps,u_0^\eps,\vartheta_0^\eps)\|_{Y^p_{0,\eps,\nu}}
\leq \eta \nu,
\end{equation}
then System \eqref{eq:NSF} with initial data $(a_0^\eps, u_0^\eps,\vartheta_0^\eps)$ has a unique global
solution $(a^\eps,u^\eps,\vartheta^\eps)$   in the space $Y^p_{\eps,\nu}$ with, for some constant $C$ independent of $\eps$ and $\nu,$
\begin{equation}\label{eq:uef}
\|(a^\eps,u^\eps,\vartheta^\eps)\|_{Y^p_{\eps,\nu}}\leq C\|(a_0^\eps,u_0^\eps,\vartheta_0^\eps)\|_{Y^p_{0,\eps,\nu}}.
\end{equation}
Furthermore, in the $\R^d$ case, if  $(a_0^\eps,u_0^\eps,\vartheta_0^\eps)$ is a family of data fulfilling \eqref{eq:smalldataf}
with  $\cP u^\eps_0\to v_0$  and $\vartheta_0^\eps-a_0^\eps \to\Theta_0$ for suitable norms, 
then we have
\begin{itemize}
\item  $(q^\eps,\cQ u^\eps)\to0$ with  $q^\eps:=\vartheta^\eps+a^\eps,$
\item $\cP u^\eps\to u$ with $u$ solution to \eqref{eq:ins},
\item $\Theta^\eps\to\Theta$ with $\Theta^\eps:=\vartheta^\eps-a^\eps$ and  $\Theta$ satisfying 
\begin{equation}\label{eq:Theta}
\d_t\Theta-\frac\kappa2\Delta\Theta+u\cdot\nabla\Theta=0,\qquad \Theta|_{t=0}=\Theta_0.
\end{equation}
\end{itemize}
More precisely, we have
\begin{equation}\label{eq:strong1}
\|(q^\eps,\cQ u^\eps)\|^{\ell,\wt\eps}_{L^2(\dot B^{(d+1)/p-1/2}_{p,1})}\lesssim \nu^{-1/2}\wt\eps^{\,1/2-1/p}\|(a_0^\eps,u_0^\eps,\vartheta_0^\eps)\|_{Y^p_{0,\eps,\nu}},
\end{equation}
 \begin{multline}\label{eq:convu}
   \|\cP u^\eps-u\|_{L^\infty(\dot B^{(d+1)/p-3/2}_{p,1})} +\mu \|\cP u^\eps-u\|_{L^1(\dot B^{(d+1)/p+1/2}_{p,1}}\\
  \lesssim \|\cP u^\eps_0-v_0\|_{\dot B^{(d+1)/p-3/2}_{p,1}}+\wt\eps^{\,1/2-1/p} \|(a_0^\eps,u_0^\eps,\vartheta_0^\eps)\|_{Y^p_{0,\eps,\nu}},
  \end{multline}
and  
   \begin{multline}\label{eq:convtheta}
  \|\dT^\eps\|_{L^\infty(\dot B^{(d+1)/p-3/2}_{p,1}+\dot B^{d/p-2}_{p,1})}
+\|\dT^\eps\|_{L^2(\dot B^{(d+1)/p-1/2}_{p,1})+L^1(\dot B^{d/p}_{p,1})}
\\\lesssim    \|\Theta^\eps_0-\Theta_0\|_{\dot B^{(d+1)/p-3/2}_{2,1}+\dot B^{d/p-2}_{p,1}}
+\wt\eps^{\,1/2-1/p}\|(a_0^\eps,u_0^\eps,\vartheta_0^\eps)\|_{Y^p_{0,\eps,\nu}}.
\end{multline}
\end{thm}
\begin{rmk} Regarding the global existence and convergence issues, 
we expect similar results for slightly larger Besov spaces, as in the barotropic case. 
Here we only considered Besov spaces \emph{with last index $1$} for simplicity, in order 
to benefit from uniqueness (see \cite{ChD}), an open question otherwise, and also because it allows us to avoid resorting to
 $\wt L^m(\dot B^s_{p,r})$ spaces and complicated product estimates.
\end{rmk}
\begin{proof}
As in the barotropic case, performing  a suitable change of unknowns reduces the proof to the case $\eps=\nu=1,$
and coefficients $\tilde\mu,$ $\tilde\lambda$ and $\tilde\kappa.$ More precisely, we set
\begin{equation}\label{eq:changefull}
(a,u,\vartheta)(t,x)=\eps(a^\eps,u^\eps,\vartheta^\eps)(\eps^2\nu t,\eps\nu x).
\end{equation} 
Thanks to \eqref{simpdecomp},  we  notice that
\begin{equation}
\nu\|(a,u,\vartheta)\|_{Y^p_{1,1}}=\|(a^\eps,u^\eps,\vartheta^\eps)\|_{Y^p_{\eps,\nu}}
\quad\hbox{and}\quad
\nu\|(a_0,u_0,\vartheta_0)\|_{Y^p_{0,1,1}}=\|(a^\eps_0,u^\eps_0,\vartheta^\eps_0)\|_{Y^p_{0,\eps,\nu}}.
\end{equation}
So we may assume from now on that $\nu=\eps=1,$ and thus omit the exponent $\eps.$
\medbreak
Let us give the outline of the proof. The first six steps are dedicated to proving global-in-time
a priori estimates (namely \eqref{eq:uef}) for smooth solutions to \eqref{eq:NSF}, which 
is a rather easy adaptation of what we did in the barotropic case. 
In Step 7, we sketch the proof of existence. The last step is dedicated
to the low Mach number asymptotics in the $\R^d$ case. 
Throughout, we assume that \eqref{eq:smalla} is satisfied,  so that 
one may freely apply Proposition \ref{p:compo}. 

\subsubsection*{Step 1. Incompressible part of the velocity}

Let $\wt\cA:=\cA/\nu.$ we have
$$
\d_t\cP u-\wt\mu\Delta\cP u=-\cP(u\cdot\nabla u)-\cP(J(a)\wt\cA u)-\cP(\vartheta\nabla K(a)) \quad\hbox{with }\ J(0)=K(0)=0.
$$
Hence heat estimates \eqref{eq:heat} yield
$$
\|\cP u\|_{L^\infty(\dot B^{d/p-1}_{p,1})\cap L^1(\dot B^{d/p+1}_{p,1})}\!\lesssim\!
\|\cP u_0\|_{\dot B^{d/p-1}_{p,1}} +\|\cP(u\cdot\nabla u)+\cP(J(a)\wt\cA u)+\cP(\vartheta\nabla(K(a)))\|_{L^1(\dot B^{d/p-1}_{p,1})}.
$$
Only the last term is new compared to the barotropic case.
Decomposing it into
$$\vartheta\nabla(K(a))=\vartheta^\ell\nabla(K(a))+\vartheta^h\nabla(K(a)),$$
we may write\begin{multline}\label{eq:poly0}
\|\cP(\vartheta\nabla(K(a)))\|_{L^1(\dot B^{d/p-1}_{p,1})}\lesssim\|\nabla(K(a))\|_{L^2(\dot B^{d/p-1}_{p,1})}
\|\vartheta^\ell\|_{L^2(\dot B^{d/p}_{p,1})}\\
+\|\nabla(K(a))\|_{L^\infty(\dot B^{d/p-1}_{p,1})}\|\vartheta^h\|_{L^1(\dot B^{d/p}_{p,1})}.
\end{multline}
So arguing as in the barotropic case and using \eqref{eq:B}, we eventually get
\begin{equation}\label{eq:poly1}
\|\cP u\|_{L^\infty(\dot B^{d/p-1}_{p,1})\cap L^1(\dot B^{d/p+1}_{p,1})}\lesssim
\|\cP u_0\|_{\dot B^{d/p-1}_{p,1}}+2^{j_0}(1+\|a\|_{L^\infty(\dot B^{d/p}_{p,1})})\|(a,u,\vartheta)\|_{Y^p_{1,1}}^2.
\end{equation}


\subsubsection*{Step 2. Low frequencies}

Applying Projector $\cQ$ to  the velocity equation, we see that $(a,\cQ u,\vartheta)$ fulfills
$$\left\{
 \begin{aligned}
      & \d_t a+\div\cQ u=-\div (au),\\
      &\d_t\cQ u -\Delta\cQ u+\nabla(a+\vartheta) =\cQ\bigl(-u\cdot \nabla u-J(a)\wt\cA u+(a-\vartheta)\nabla(K(a))\bigr),\\
        &\d_t\vartheta-\wt\kappa\Delta\vartheta+\div\cQ u=-\div(\vartheta u)-\wt\kappa J(a)\Delta\vartheta
        +\frac1{1+a}\Bigl(2\wt\mu|Du|^2+\wt\lambda(\div u)^2\Bigr)\cdotp
        \end{aligned}
  \right.
$$
The results of \cite{D2} guarantee that
\begin{equation}\label{eq:poly1a}
\|(a,\cQ u,\vartheta)\|^\ell_{L^\infty(\dot B^{d/2-1}_{2,1})\cap L^1(\dot B^{d/2+1}_{2,1})}
\lesssim \|(a_0,\cQ u_0,\vartheta_0)\|^\ell_{\dot B^{d/2-1}_{2,1}}
+\|r.h.s.\|^\ell_{L^1(\dot B^{d/2-1}_{2,1})}.
\end{equation}
Compared to the barotropic case, we have to bound in $L^1(\dot B^{d/2-1}_{2,1})$ the low frequencies
of the following additional terms:
\begin{equation}\label{eq:poly1b}
\vartheta\nabla (K(a)), \quad\div(\vartheta u),\quad\wt\kappa J(a)\Delta\vartheta
        \quad \hbox{and }\quad \frac1{1+a}\Bigl(2\wt\mu|Du|^2+\wt\lambda(\div u)^2\Bigr)\cdotp
\end{equation}
To handle the first term, we start with the observation that 
\begin{equation}\label{eq:Ka0}
\|\nabla(K(a))\|_{\dot B^{d/2-1}_{2,1}}^\ell\lesssim \bigl(1+\|a\|_{\dot B^{d/p}_{p,1}}\bigr)\bigl(\|a\|_{\dot B^{d/2}_{2,1}}^\ell
+\|a\|_{\dot B^{d/p}_{p,1}}^h\bigr)+2^{j_0}\|a\|_{\dot B^{d/p-1}_{p,1}}\|a\|_{\dot B^{d/p}_{p,1}}^h.
\end{equation}
Indeed, because $\nabla((K(a))=K'(0)\nabla a+ \wt K(a) \nabla a$ for some smooth function $\wt K$ vanishing at zero, it suffices to prove that\footnote{For $p^*,$ we keep the definition  $1/p+1/p^*=1/2.$}
$$
\|\wt K(a) \nabla a\|_{\dot B^{d/2-1}_{2,1}}^\ell\lesssim 
\|a\|_{\dot B^{d/p}_{p,1}}\bigl(\|\nabla a\|_{\dot B^{d/2-1}_{2,1}}^\ell
+\|\nabla a\|_{\dot B^{d/p^*-1}_{p^*,1}}\bigr)+2^{j_0}\|a\|_{\dot B^{d/p-1}_{p,1}}\|\nabla a\|_{\dot B^{d/p-1}_{p,1}}^h.
$$
To this end, we use Bony's decomposition restricted to low frequencies:
$$
(\wt K(a) \nabla a)^\ell=(T_{\nabla a} \wt K(a))^\ell +(R(\nabla a, \wt K(a)))^\ell+  T_{(\wt K(a))^\ell} \nabla a^\ell
+(\dot S_{j_0} \wt K(a)\dot\Delta_{j_0+1}\nabla a)^\ell.
$$
To deal with the first two terms, we just use  \eqref{eq:Qu4a}.
For the third one, we use that $T:L^\infty\times\dot B^{d/2-1}_{2,1}\to\dot B^{d/2-1}_{2,1}$ and
 the embedding $\dot B^{d/p}_{p,1}\hookrightarrow L^\infty.$ 
For the last one, we argue as follows:
$$
2^{j_0(d/2-1)}\|\dot S_{j_0} \wt K(a)\dot\Delta_{j_0+1}\nabla a\|_{L^2} \leq 2^{j_0}\, 
2^{j_0(d/p^*-1)}\|\dot S_{j_0}\wt K(a)\|_{L^{p^*}}\,2^{j_0(d/p-1)}\|\dot\Delta_{j_0+1}\nabla a\|_{L^p}.
$$
Putting all those inequalities together, and using also
composition estimates and the fact that $d/p^*-1\leq0$ eventually leads to the desired inequality.
\smallbreak
Let us now bound $(\vartheta\nabla(K(a)))^\ell$ in $L^1(\dot B^{d/2-1}_{2,1}).$
We start again from  Bony's decomposition:
\begin{multline}\label{eq:bonydec}
(\vartheta\nabla (K(a)))^\ell=(T_{\nabla(K(a))}\vartheta)^\ell+(R(\nabla(K(a)),\vartheta))^\ell
\\+T_{\vartheta^\ell}\nabla(K(a))^\ell+(\dot S_{j_0} \vartheta\dot\Delta_{j_0+1}\nabla K(a))^\ell.
\end{multline}
The first two terms may be bounded by splitting $\vartheta$ into $\vartheta^\ell+\vartheta^h,$
using the continuity of $R$ and $T$ from $\dot B^{d/p-1}_{p,1}\times\dot B^{d/p}_{p,1}$ to $\dot B^{d/2-1}_{2,1}$ .
We end up with 
$$\displaylines{\quad
\| (T_{\nabla(K(a))}\vartheta)^\ell\|_{L^1(\dot B^{d/2-1}_{2,1})} \lesssim
\|\nabla(K(a))\|_{L^2(\dot B^{d/p-1}_{p,1})}\bigl(\|\vartheta^\ell\|_{L^2(\dot B^{d/p}_{p,1})}
+2^{j_0}\|\vartheta^h\|_{L^2(\dot B^{d/p-1}_{p,1})}\bigr).\quad}
$$ and
 \beno
\| (R(\nabla(K(a)),\vartheta))^\ell\|_{L^1(\dot B^{d/2-1}_{2,1})} &\lesssim& \|\nabla(K(a))\|_{L^\infty(\dot B^{d/p-1}_{p,1})} \|\vartheta^h\|_{L^1(\dot B^{d/p}_{p,1})}\\&&
+
\|\nabla(K(a))\|_{L^2(\dot B^{d/p-1}_{p,1})} \|\vartheta^\ell\|_{L^2(\dot B^{d/p}_{p,1})}.
\eeno
For the third term in \eqref{eq:bonydec}, by virtue of \eqref{eq:Ka0}, we write
\beno
\|T_{\vartheta^\ell}\nabla(K(a))^\ell\|_{L^1(\dot B^{d/2-1}_{2,1})}&\lesssim &\|\vartheta^\ell\|_{L^2(L^\infty)}\|\nabla(K(a))^\ell\|_{\dot B^{d/2-1}_{2,1}}\\
&\lesssim & \|\vartheta^\ell\|_{L^2(\dot B^{d/2}_{2,1})}\bigl(1+2^{j_0}\|a^\ell\|_{L^\infty(\dot B^{d/2-1}_{2,1})}\\&&+\|a^h\|_{L^\infty(\dot B^{d/p}_{p,1})}\bigr)\bigl( \|a\|^\ell_{L^2(\dot B^{d/2}_{2,1})}+\|a\|^h_{L^2(\dot B^{d/p}_{p,1})}\bigr).
\eeno
Finally,
$$
2^{j_0(d/2-1)}\|\dot S_{j_0} \vartheta\dot\Delta_{j_0+1}\nabla K(a))\|_{L^2}
 \leq 2^{j_0(d/p^*-1)}\|\dot S_{j_0}\vartheta\|_{L^{p^*}}\,2^{j_0d/p}\|\dot\Delta_{j_0+1}\nabla(K(a))\|_{L^p}.
$$
Hence
$$
\|\dot S_{j_0} \vartheta\dot\Delta_{j_0+1}\nabla(K(a))\|_{L^1(\dot B^{d/2-1}_{2,1})}
 \lesssim 2^{j_0} \|\vartheta^\ell\|_{L^\infty(\dot B^{d/p^*-1}_{p^*,1})} 
 \|K(a)\|^h_{L^1(\dot B^{d/p}_{p,1})}.
 $$
 That the last term does belong to $L^1(\dot B^{d/p}_{p,1})$ may be seen by writing 
 $$
 K(a)=K'(0)\, a+ \wt K(a)\, a\quad\hbox{with }\ \wt K(0)=0,
 $$
 which ensures, using composition estimates in $\dot B^{d/p}_{p,1},$
 \begin{equation}\label{eq:Ka}
  \|K(a)\|_{L^1(\dot B^{d/p}_{p,1}))}^h\lesssim \|a\|^h_{L^1(\dot B^{d/p}_{p,1})}
  +\|a\|_{L^2(\dot B^{d/p}_{p,1})}^2.
  \end{equation}
  Resuming to \eqref{eq:bonydec}, we conclude that
  \begin{multline}
  \|(\vartheta\nabla (K(a)))^\ell\|_{L^1(\dot B^{d/2-1}_{2,1})}\lesssim
  2^{j_0}\|\vartheta^\ell\|_{L^\infty(\dot B^{d/2-1}_{2,1})}\bigl(\|a^h\|_{L^1(\dot B^{d/p}_{p,1})}+\|a\|_{L^2(\dot B^{d/p}_{p,1})}^2\bigr)\\
  +\bigl(1+2^{j_0}\|a^\ell\|_{L^\infty(\dot B^{d/2-1}_{2,1})}+\|a^h\|_{L^\infty(\dot B^{d/p-1}_{p,1})}\bigr)\bigl(
  \|\vartheta^\ell\|_{L^2(\dot B^{d/2}_{2,1})}\\+2^{j_0}\|\vartheta^h\|_{L^2(\dot B^{d/p-1}_{p,1})}+ \|\vartheta^h\|_{L^1(\dot B^{d/p}_{p,1})}\bigr)\bigl( \|a\|^\ell_{L^2(\dot B^{d/2}_{2,1})}+\|a\|^h_{L^2(\dot B^{d/p}_{p,1})}\bigr).
    \end{multline}
To handle $\div(\vartheta u),$ we decompose $\vartheta$ into low and high frequencies. 
To deal with both parts, we resort again to Bony's decomposition and continuity results for $R$ and $T.$
We end up with 
\begin{eqnarray}\label{eq:poly1c}
\|\vartheta^\ell u\|_{\dot B^{d/2}_{2,1}}\lesssim \|\vartheta^\ell\|_{\dot B^{d/p^*-1}_{p^*,1}} \|u\|_{\dot B^{d/p+1}_{p,1}}
+\|u\|_{L^\infty}\|\vartheta^\ell\|_{\dot B^{d/2}_{2,1}},\\\label{eq:poly1d}
\|\vartheta^h u\|_{\dot B^{d/2-1}_{2,1}}\lesssim \|\vartheta^h\|_{\dot B^{d/p-1}_{p,1}} \|u\|_{\dot B^{d/p}_{p,1}}
+\|u\|_{\dot B^{d/p^*-1}_{p^*,1}}\|\vartheta^h\|_{\dot B^{d/p}_{p,1}}.
\end{eqnarray}
Therefore, taking advantage of the low frequency cut-off and of Bernstein inequality yields
  \begin{multline}
  \|(\div(\vartheta u))^\ell\|_{L^1(\dot B^{d/2-1}_{2,1})}\lesssim 
  \|\vartheta^\ell\|_{L^\infty(\dot B^{d/2-1}_{2,1})}\|u\|_{L^1(\dot B^{d/p+1}_{p,1})}
  \\+\bigl(\|\vartheta^\ell\|_{L^2(\dot B^{d/2}_{2,1})}+2^{j_0}\|\vartheta^h\|_{L^2(\dot B^{d/p-1}_{p,1})}\bigr)
  \|u\|_{L^2(\dot B^{d/p}_{p,1})}+2^{j_0}\|u\|_{L^\infty(\dot B^{d/p-1}_{p,1})}\|\vartheta^h\|_{L^1(\dot B^{d/p}_{p,1})}.
  \end{multline}
For the next term, we use 
$$
J(a)\Delta\vartheta=J(a)\Delta\vartheta^\ell+J(a)\Delta\vartheta^h
$$
and Bony's decomposition. 
For the first term, we easily get
$$
\|J(a)\Delta\vartheta^\ell\|_{L^1(\dot B^{d/2-1}_{2,1})}\lesssim 
\|\Delta\vartheta^\ell\|_{L^1(\dot B^{d/2-1}_{2,1})}\|a\|_{L^\infty(\dot B^{d/p}_{p,1})}.
$$
For the second one, we use that
$R$ and $T$ map $\dot B^{d/p-2}_{p,1}\times \dot B^{d/p}_{p,1}$ to $\dot B^{d/2-2}_{2,1},$ if $p<d$ and $d\geq3,$
and that 
$$
\|(T_{J(a)}\Delta\vartheta^h+R(J(a),\Delta\vartheta^h))^\ell\|_{\dot B^{d/2-2}_{2,1}}\lesssim \|J(a)\|_{\dot B^{d/p}_{p,1}}\|\Delta\vartheta^h\|_{\dot B^{d/p-2}_{p,1}}.
$$
Hence, combining with  Bernstein inequality, 
$$\displaylines{
\|(J(a)\Delta\vartheta)^\ell\|_{L^1(\dot B^{d/2-1}_{2,1})}\hfill\cr\hfill\lesssim 2^{j_0}\bigl(\|a\|_{L^\infty(\dot B^{d/p}_{p,1})}
+2^{j_0}\|a\|_{L^\infty(\dot B^{d/p-1}_{p,1})}\bigr)
\bigl(\|\Delta\vartheta^\ell\|_{L^1(\dot B^{d/2-1}_{2,1})}+\|\Delta\vartheta^h\|_{L^1(\dot B^{d/p-2}_{p,1})}\bigr).}
$$
To handle the last term in \eqref{eq:poly1b}, we use the fact that
$$\begin{array}{rcl}
\|T_{J(a)}Du\otimes Du\|_{\dot B^{d/2-2}_{2,1}}&\lesssim&
 \|J(a)\|_{L^\infty}\| Du\otimes Du\|_{\dot B^{d/2-2}_{2,1}}\\
 \|R(J(a),Du\otimes Du)\|_{\dot B^{d/2-2}_{2,1}}&\lesssim&
 \|J(a)\|_{\dot B^{d/p}_{p,1}}\|Du\otimes Du\|_{\dot B^{d/2-2}_{2,1}}\\
 \|T_{Du\otimes Du} J(a)\|_{\dot B^{d/2-2}_{2,1}}&\lesssim& \|Du\otimes Du\|_{\dot B^{d/p^*-2}_{p^*,1}} \|J(a)\|_{\dot B^{d/p}_{p,1}}.
\end{array}
$$
At this point, we notice that, under assumption $2\leq p\leq 2d/(d-2),$ $p<d$ and $d\geq3,$  the usual product
maps $\dot B^{d/p-1}_{p,1}\times\dot B^{d/p-1}_{p,1}$ to $\dot B^{d/2-2}_{2,1}.$
Therefore
\begin{equation}
\Bigl\|\frac1{1+a}Du\otimes Du\Bigr\|_{L^1(\dot B^{d/2-2}_{2,1})}
\lesssim \bigl(1+\|a\|_{L^\infty(\dot B^{d/p}_{p,1})}\bigr)\|Du\|_{L^2(\dot B^{d/p-1}_{p,1})}^2.
\end{equation}
Inserting all the above inequalities in \eqref{eq:poly1a} and using \eqref{eq:B}, we thus end up with
\begin{multline}\label{eq:poly2}
\|(a,\cQ u,\vartheta)\|^\ell_{L^\infty(\dot B^{d/2-1}_{2,1})\cap L^1(\dot B^{d/2+1}_{2,1})}
\\\lesssim\|(a_0,\cQ u_0,\vartheta_0)\|^\ell_{\dot B^{d/2-1}_{2,1}} +2^{2j_0}(1+ \|(a,u,\vartheta)\|_{Y^p_{1,1}})\|(a,u,\vartheta)\|_{Y^p_{1,1}}^2.
\end{multline}


\subsubsection*{Step 3. High frequencies: the effective velocity} Let $w:=\cQ u+(-\Delta)^{-1}\nabla a.$ We have
$$
\d_tw-\Delta w=-\cQ(u\cdot\nabla u)+\cQ(J(a)\wt\cA u)-\cQ(\vartheta \nabla K(a))+a\nabla K(a)+\cQ(au)-\nabla\vartheta+w-(-\Delta)^{-1}\nabla a.
$$
By virtue of \eqref{eq:heat}, we have
$$
\|w\|^h_{L^\infty(\dot B^{d/p-1}_{p,1})\cap L^1(\dot B^{d/p+1}_{p,1})}\lesssim \|w_0\|^h_{\dot B^{d/p-1}_{p,1}}
+ \|r.h.s.\|^h_{L^1(\dot B^{d/p-1}_{p,1})}.
$$
Compared to the barotropic case,  two new terms have to be handled : $\cQ(\vartheta\nabla (K(a)))$  and $\nabla\vartheta.$
The first one has been estimated in \eqref{eq:poly0}, and the second one is just linear.
 We eventually get if $j_0$ is large enough:
 \begin{multline}\label{eq:poly3}
 \|w\|^h_{L^\infty(\dot B^{d/p-1}_{p,1})\cap L^1(\dot B^{d/p+1}_{p,1})}\lesssim \|w_0\|^h_{\dot B^{d/p-1}_{p,1}}
\\+2^{j_0}\|(a,u,\vartheta)\|_{Y^p_{1,1}}^2+\|\nabla\vartheta\|_{L^1(\dot B^{d/p-1}_{p,1})}^h+2^{-2j_0}\|a\|_{L^1(\dot B^{d/p}_{p,1})}^h.
 \end{multline}

\subsubsection*{Step 4. High frequencies: the temperature}

Applying \eqref{eq:heat} to the heat equation
$$
\d_t\vartheta-\wt\kappa\Delta\vartheta=-a-\div w-\div(\vartheta u)-\wt\kappa J(a)\Delta\vartheta
        +\frac1{1+a}\Bigl(2\wt\mu|Du|^2+\wt\lambda(\div u)^2\Bigr)
$$
yields
$$
 \|\vartheta\|^h_{L^\infty(\dot B^{d/p-2}_{p,1})\cap L^1(\dot B^{d/p}_{p,1})}\lesssim \|\vartheta_0\|^h_{\dot B^{d/p-2}_{p,1}}
 +\|r.h.s.\|_{L^1(\dot B^{d/p-2}_{p,1})}^h.
  $$
  The term $\div(\vartheta u)$ can be bounded according to \eqref{eq:poly1c} and \eqref{eq:poly1d}, 
  using obvious embedding. 
  For the other nonlinear terms, we observe that under condition $p<d,$ we have
  $$\begin{array}{lll}
  \|J(a)\Delta\vartheta^h\|_{L^1(\dot B^{d/p-2}_{p,1})}&\!\!\!\lesssim\!\!\!& \|a\|_{L^\infty(\dot B^{d/p}_{p,1})}\|\Delta\vartheta^h\|_{L^1(\dot B^{d/p-2}_{p,1})},\\[2ex]
    \|J(a)\Delta\vartheta^\ell\|_{L^1(\dot B^{d/p-2}_{p,1})}^h&\!\!\!\lesssim\!\!\!&2^{-j_0}   \|J(a)\Delta\vartheta^\ell\|_{L^1(\dot B^{d/p-1}_{p,1})}^h
    \lesssim2^{-j_0}
     \|a\|_{L^\infty(\dot B^{d/p}_{p,1})}\|\Delta\vartheta^\ell\|_{L^1(\dot B^{d/2-1}_{2,1})},
     \end{array}
     $$
     $$
 \Bigl\| \frac1{1+a}\Bigl(2\mu|Du|^2\!+\!\lambda(\div u)^2\Bigr)\Bigr\|_{L^1(\dot B^{d/p-2}_{p,1})}
 \lesssim (1+ \|a\|_{L^\infty(\dot B^{d/p}_{p,1})})\|\nabla u\|_{L^2(\dot B^{d/p-1}_{p,1})}^2,
 $$
whence
  \begin{multline}\label{eq:poly4}
   \|\vartheta\|^h_{L^\infty(\dot B^{d/p-2}_{p,1})\cap L^1(\dot B^{d/p}_{p,1})}\lesssim \|\vartheta_0\|^h_{\dot B^{d/p-2}_{p,1}}
   +2^{-2j_0}\|a+\div w\|_{L^1(\dot B^{d/p}_{p,1})}^h
  \\ +2^{j_0}(1+ \|(a,u,\vartheta)\|_{Y^p_{1,1}})\|(a,u,\vartheta)\|_{Y^p_{1,1}}^2.
  \end{multline}


  \subsubsection*{Step 5. High frequencies: the density}

Exactly as in the barotropic case, Inequality \eqref{eq:ah} is fulfilled.


\subsubsection*{Step 6. Closure of the estimates}
Inserting \eqref{eq:poly4} in \eqref{eq:poly3}, we get for large enough $j_0$
$$\displaylines{
 \|w\|^h_{L^\infty(\dot B^{d/p-1}_{p,1})\cap L^1(\dot B^{d/p+1}_{p,1})}\lesssim \|w_0\|^h_{\dot B^{d/p-1}_{p,1}}
 + \|\vartheta_0\|^h_{\dot B^{d/p-2}_{p,1}} \hfill\cr\hfill
+2^{j_0}(1+  \|(a,u,\vartheta)\|_{Y^p_{1,1}}) \|(a,u,\vartheta)\|_{Y^p_{1,1}}^2
+2^{-2j_0}\|a\|^h_{L^1(\dot B^{d/p}_{p,1})}.}
 $$
Next, plugging that latter inequality in \eqref{eq:ah},    we get for large enough $j_0,$
$$
\|a\|^h_{L^1\cap L^\infty(\dot B^{d/p}_{p,1})} \lesssim \|a_0\|^h_{\dot B^{d/p}_{p,1}}+
\|w_0\|^h_{\dot B^{d/p-1}_{p,1}}
 + \|\vartheta_0\|^h_{\dot B^{d/p-2}_{p,1}}
+2^{j_0}(1+  \|(a,u,\vartheta)\|_{Y^p_{1,1}}) \|(a,u,\vartheta)\|_{Y^p_{1,1}}^2.
$$
Resuming to \eqref{eq:poly1} and \eqref{eq:poly2}, it is now easy to conclude that
$$
  \|(a,u,\vartheta)\|_{Y^p_{1,1}}\lesssim \|(a_0,u_0,\vartheta_0)\|_{Y^p_{0,1,1}}
  + 2^{2j_0}(1+  \|(a,u,\vartheta)\|_{Y^p_{1,1}}) \|(a,u,\vartheta)\|_{Y^p_{1,1}}^2,
  $$
  from which it is clear that we may get \eqref{eq:uef} if $ \|(a_0,u_0,\vartheta_0)\|_{Y^p_{0,1,1}}$ is small enough.

  
  \subsubsection*{Step 7. The proof of global existence and uniqueness}
  
 Uniqueness up to $p<d$ is just a consequence of the recent paper \cite{ChD}.
Local-in-time existence of a solution $(a,u,\vartheta)$ to \eqref{eq:NSF}  with $a\in\cC([0,T];\dot B^{d/p}_{p,1}),$
$u\in\cC([0,T];\dot B^{d/p-1}_{p,1})\cap L^1(0,T;\dot B^{d/p+1}_{p,1})\ \hbox{ and }\ 
\vartheta\in\cC([0,T];\dot B^{d/p-2}_{p,1})\cap L^1(0,T;\dot B^{d/p}_{p,1})$
has been established in \cite{D0}. 
 That the additional low frequency $L^2$ type regularity is preserved during the evolution 
is a consequence of the computations that have been carried out in Step 2. 

Finally,  by slight modifications of the blow-up criterion of Prop. 10.10 of \cite{BCD}, one can show that if 
$$
\|\nabla u\|_{L^1_T(L^\infty)}+\|a\|_{L_T^\infty(\dot B^{d/p}_{p,1})}+\|\vartheta\|_{L_T^1(\dot B^{d/p}_{p,1})}<\infty
$$
then the solution may be continued beyond $T.$
As the norm in the space $Y^p_{1,1}$ (restricted to $[0,T)$) clearly controls the above l.h.s.,  Inequality  
\eqref{eq:uef} implies the global existence.

\subsubsection*{Step 8. Low Mach number limit : strong convergence in the whole space case} 

As in our recent work \cite{DH} dedicated to the Oberbeck-Boussinesq approximation, 
the proof of strong convergence relies on the dispersive properties of the system 
fulfilled by $q^\eps:=\vartheta^\eps+a^\eps$ and $\cQ u^\eps,$ namely
$$
\left\{\begin{array}{l}
\d_tq^\eps+\Frac2\eps\div\cQ u^\eps=-\div(u^\eps q^\eps)+\kappa\Delta\vartheta^\eps
+\kappa J(\eps a^\eps)\Delta\vartheta^\eps+\frac\eps{1\!+\!\eps a^\eps}
\bigl(2\mu |Du^\eps|^2+\lambda(\div u^\eps)^2\bigr),\\[1.5ex]
\d_t\cQ u^\eps+\Frac1\eps\nabla q^\eps=\nu\Delta\cQ u^\eps-\cQ(u^\eps\cdot\nabla u^\eps)
-\cQ(J(\eps a^\eps)\cA u^\eps)+\cQ\biggl((a^\eps-\vartheta^\eps)\frac{\nabla a^\eps}{1\!+\!\eps a^\eps}\biggr)\cdotp
\end{array}\right.
$$
Remembering  that the low frequencies of the r.h.s. have been bounded in $L^1(\R_+;\dot B^{d/2-1}_{2,1})$
by $C_0^{\eps,\nu}:=\|(a_0^\eps,u^\eps_0,\vartheta_0^\eps)\|_{Y^p_{0,\eps,\nu}}$ (see  Step 2), one can mimic the proof of the strong
convergence for the barotropic case in the case $d\geq3$ (see the beginning of Section \ref{s:strong})
and easily conclude that \eqref{eq:strong1} is satisfied. 

The high frequencies of $a^\eps$ and $\cQ u^\eps$ may be bounded as in \eqref{eq:strong1}
(argue as in the barotropic case)
but not  $(\vartheta^\eps)^{h,\wt\eps}$ which is one derivative less regular than $(a^\eps)^{h,\wt\eps}.$
\medbreak
Let us now study the strong convergence of $\cP u^\eps$ to $u.$
To this end, we observe that $\du^\eps:=\cP u^\eps-u$ fulfills
\begin{multline}\label{eq:strong0}
\d_t\du^\eps-\mu\Delta\du^\eps+\cP(\cP u^\eps\!\cdot\!\nabla\du^\eps+\du^\eps\!\cdot\!\nabla u)
+\cP\bigl(u^\eps\!\cdot\!\nabla\cQ u^\eps+\cQ u^\eps\!\cdot\!\nabla\cP u^\eps)+J(\eps a^\eps)\cA u^\eps\bigr)\\
=\cP\biggl(\frac1{1+\eps a^\eps}\biggl(
\nabla (q^\eps)^{\ell,\wt\eps}a^\eps+\nabla (a^\eps)^{h,\wt\eps}a^\eps
+\nabla (\vartheta^\eps)^{h,\wt\eps}a^\eps\biggr)+J(\eps a^\eps)\nabla(a^\eps\vartheta^\eps)\biggr)\cdotp
\end{multline}
The first line may be handled as in the barotropic case : we get
$$
\displaylines{\|\cP(\cP u^\eps\cdot\nabla\du^\eps+\du^\eps\cdot\nabla u)\|_{L^1(\dot B^{(d+1)/p-3/2}_{p,1})}
\lesssim \|\cP u^\eps\|_{L^\infty(\dot B^{d/p-1}_{p,1})}\|\nabla\du^\eps\|_{L^1(\dot B^{(d+1)/p-1/2}_{p,1})}
\hfill\cr\hfill+\|\du^\eps\|_{L^\infty(\dot B^{(d+1)/p-3/2}_{p,1})}\|\nabla u\|_{L^1(\dot B^{d/p}_{p,1})},\cr
\|\cP\bigl(u^\eps\cdot\nabla\cQ u^\eps+\cQ u^\eps\cdot\nabla\cP u^\eps)+J(\eps a^\eps)\cA u^\eps\bigr)\|_{L^1(\dot B^{(d+1)/p-3/2}_{p,1})}
\lesssim\nu^{-1}\eps^{\,1/2-1/p}(1+\nu^{-1}C_0^{\eps,\nu})(C_0^{\eps,\nu})^2.}
$$
In order to bound the terms of the second line of \eqref{eq:strong0}, we  shall use repeatedly the fact that  for any smooth function $K$ vanishing at $0,$ we have, by virtue of Proposition \ref{p:compo},
\begin{equation}\label{eq:strong2}
\|K(\eps a^\eps)\|_{L^\infty(\dot B^{d/p}_{p,1})}\lesssim\nu^{-1}\bigl(\|a^\eps\|^{\ell,\wt\eps}_{L^\infty(\dot B^{d/p-1}_{p,1})}
+\wt\eps\|a^\eps\|^{h,\wt\eps}_{L^\infty(\dot B^{d/p}_{p,1})}\bigr)\lesssim\nu^{-1}C_0^{\eps,\nu}.
\end{equation}
On the one hand, using product laws in Besov spaces yields
$$\begin{array}{lll}
\|\nabla(q^\eps)^{\ell,\wt\eps} a^\eps\|_{L^1(\dot B^{(d+1)/p-3/2}_{p,1})}
&\lesssim& \|(q^\eps)^{\ell,\wt\eps}\|_{L^2(\dot B^{(d+1)/p-1/2}_{p,1})}\| a^\eps\|_{L^2(\dot B^{d/p}_{p,1})},\\[1ex]
\|\nabla(a^\eps)^{h,\wt\eps} a^\eps\|_{L^1(\dot B^{(d+1)/p-3/2}_{p,1})}
&\lesssim& \|(a^\eps)^{h,\wt\eps}\|_{L^2(\dot B^{(d+1)/p-1/2}_{p,1})}\|a^\eps\|_{L^2(\dot B^{d/p}_{p,1})},\\[1ex]
\|\nabla(\vartheta^\eps)^{h,\wt\eps} a^\eps\|_{L^1(\dot B^{(d+1)/p-3/2}_{p,1})}
&\lesssim& \|(\vartheta^\eps)^{h,\wt\eps}\|_{L^1(\dot B^{d/p}_{p,1})}\| a^\eps\|_{L^\infty(\dot B^{(d+1)/p-1/2}_{p,1})}\\
[1ex] \|J(\eps a^\eps)\nabla(a^\eps\vartheta^\eps)\|_{L^1(\dot B^{(d+1)/p-3/2}_{p,1})}
&\lesssim& \|\eps a^ \eps\|_{L^\infty(\dot B^{(d+1)/p-1/2}_{p,1})}(\|(\vartheta^\eps)^{h,\wt\eps}\|_{L^1(\dot B^{d/p}_{p,1})}\| a^\eps\|_{L^\infty(\dot B^{d/p}_{p,1})}\\[1ex]&&+\|(\vartheta^\eps)^{l,\wt\eps}\|_{L^2(\dot B^{d/p}_{p,1})}\| a^\eps\|_{L^2(\dot B^{d/p}_{p,1})}).
\end{array}
$$
Hence using \eqref{eq:conv4}, \eqref{eq:strong1}, \eqref{eq:uef} and \eqref{eq:strong2},  
$$
\displaylines{
\Bigl\|\cP\Bigl(\frac1{1+\eps a}\Bigl(
\nabla (q^\eps)^{\ell,\wt\eps} a^\eps+\nabla(a^\eps)^{h,\wt\eps} a^\eps
+\nabla(\vartheta^\eps)^{h,\wt\eps}a^\eps\Bigr)+J(\eps a^\eps)\nabla(a^\eps\vartheta^\eps)\Bigr)\Bigr\|_{L^1(\dot B^{(d+1)/p-3/2}_{p,1})}\hfill\cr\hfill
\lesssim \nu^{-1}(1+\nu^{-1}C_0^{\eps,\nu})(C_0^{\eps,\nu})^2.}
$$
Putting together all the above inequalities and the uniform estimate \eqref{eq:uef}, we end up with
 $$\displaylines{
\dU^\eps:= \|\du^\eps\|_{L^\infty(\dot B^{(d+1)/p-3/2}_{p,1})}
 +\mu \|\du^\eps\|_{L^1(\dot B^{(d+1)/p+1/2}_{p,1})}
 \hfill\cr\hfill\lesssim \|\cP u^\eps_0-v_0\|_{\dot B^{(d+1)/p-3/2}_{p,1}}
 +\nu^{-1}C_0^{\eps,\nu}\dU^\eps+\nu^{-1}\wt\eps^{\,1/2-1/p} (1+\nu^{-1}C_0^{\eps,\nu})(C_0^{\eps,\nu})^2,}
 $$
 which obviously implies  \eqref{eq:convu}, owing to the smallness  condition satisfied by $C_0^{\eps,\nu}.$ 
 \medbreak
Let us finally study the strong convergence of $\Theta^\eps:=\vartheta^\eps-a^\eps$
to the solution $\Theta$ of \eqref{eq:Theta}.
Given the uniform bounds for $(\vartheta^\eps_0)$ and for $(a^\eps_0),$ it is natural to 
assume that the limit $\Theta_0$ belongs to 
$\dot B^{d/2-1}_{2,1}$ (as a matter of fact  $\dot B^{d/p-1}_{p,1}$ is enough for what follows).  Likewise, as $(\cP u_0^\eps)$ is bounded in 
$\dot B^{d/p-1}_{p,1},$ one may assume that its weak limit $v_0$ belongs to $\dot B^{d/p-1}_{p,1}.$ 
Hence the corresponding solution $u$ to \eqref{eq:ins} is in  
$\cC(\R_+;\dot B^{d/p-1}_{p,1})\cap L^1(\R_+;\dot B^{d/p+1}_{p,1}),$
and using the fact that $\div(u\Theta)=u\cdot\nabla\Theta,$ it is easy to prove that the linear equation \eqref{eq:Theta} admits
 a unique solution $\Theta\in\cC_b(\R_+;\dot B^{d/2-1}_{2,1})\cap L^1(\R_+;\dot B^{d/2+1}_{2,1}).$
 \smallbreak
Next,  from \eqref{eq:Theta}, observing that $\Delta\vartheta^\eps=\frac12\Delta\Theta^\eps+\frac12\Delta q^\eps,$
we readily get that $\dT^\eps:=\Theta^\eps-\Theta$ satisfies
\begin{multline}\label{eq:strong3}
\d_t\dT^\eps-\frac\kappa2\Delta\dT^\eps=-\cP u^\eps\cdot\nabla\dT^\eps-\du^\eps\cdot\nabla\Theta
-\div(\cQ u^\eps\Theta^\eps)\\
-\frac\kappa2\Delta q^\eps-\kappa J(\eps a^\eps)\Delta\vartheta^\eps+\frac\eps{1+\eps a^\eps}
\bigl(2\mu|Du^\eps|^2+\lambda(\div u^\eps)^2\bigr).
\end{multline}
The level of regularity on which estimates for $\dT^\eps$ may be proved, is essentially given by the available estimates
for $\du^\eps,$ through the term $\du^\eps\cdot\nabla\Theta=\div(\du^\eps\Theta),$
by the fact that decay estimates are available  for the low frequencies of the term $\Delta q^\eps$ in
the space  $L^2(\R_+;\dot B^{(d+1)/p-5/2}_{p,1})$  only through \eqref{eq:strong1}, and 
by observing that the high frequencies of $\Delta q^\eps$ (and more precisely of $\Delta\vartheta^\eps$)
 are at most in the space $L^1(\dot B^{d/p-2}_{p,1}),$ but have decay $\eps.$

As regards $\du^\eps\cdot\nabla\Theta,$  product laws in Besov spaces give the following bound:
$$
\|\du^\eps\cdot\nabla\Theta\|_{L^1(\dot B^{(d+1)/p-3/2}_{p,1})}
\lesssim  \|\du^\eps\|_{L^\infty(\dot B^{(d+1)/p-3/2}_{p,1})}\|\Theta\|_{L^1(\dot B^{d/2+1}_{2,1})}.
$$
Note that  only an $L^2$-in-time estimate is available for  $(\Delta q^\eps)^{\ell,\wt\eps},$
through \eqref{eq:strong1}. However,  a small variation on \eqref{eq:heat} (see e.g. \cite{BCD}) ensures that the solution to 
$$
\d_tz-\kappa\Delta z=-\frac\kappa2(\Delta q^\eps)^{\ell,\wt\eps},\qquad z|_{t=0}=0
$$
 belongs to $\cC_b(\R_+;\dot B^{(d+1)/p-3/2}_{p,1})\cap L^2(\R_+;\dot B^{(d+1)/p-1/2}_{p,1})$
and satisfies 
$$
\|z\|_{L^2(\dot B^{(d+1)/p-1/2}_{p,1})}+ \|z\|_{L^\infty(\dot B^{(d+1)/p-3/2}_{p,1})}
\lesssim \|\Delta q^\eps\|_{L^2(\dot B^{(d+1)/p-5/2}_{p,1})}^{\ell,\wt\eps}. 
 $$
 So in short we expect to be able to bound  $\dT^\eps$ 
in $$L^\infty(\R_+;\dot B^{(d+1)/p-3/2}_{p,1}+\dot B^{d/p-2}_{p,1})\cap \bigl(L^2(\R_+;\dot B^{(d+1)/p-1/2}_{p,1})
+L^1(\R_+;\dot B^{d/p}_{p,1})\bigr).$$
Let us now look at the other terms in the r.h.s. of \eqref{eq:strong3}. 
It is clear that $(\Delta a^\eps)^{h,\wt\eps}$ may be bounded exactly as $(\Delta q^\eps)^{\ell,\wt\eps}.$
Next, product laws easily give that
$$
\begin{array}{lll}
\|\cP u^\eps\cdot\nabla\dT^\eps\|_{L^1(\dot B^{(d+1)/p-3/2}_{p,1}+\dot B^{d/p-2}_{p,1})}
&\!\!\!\!\lesssim\!\!\!\!& \|\cP u^\eps\|_{L^2(\dot B^{d/p}_{p,1})}\|\dT^\eps\|_{L^2(\dot B^{(d+1)/p-1/2}_{p,1}+\dot B^{d/p-1}_{p,1})},\\[1.5ex]
\|\div(\cQ u^\eps\Theta^\eps)\|_{L^1(\dot B^{(d+1)/p-3/2}_{p,1}+\dot B^{d/p-2}_{p,1})}&\!\!\!\!\lesssim\!\!\!\!&
\|\cQ u^\eps\|_{L^2(\dot B^{(d+1)/p-1/2}_{p,1})}\bigl(\|\Theta^\eps\|_{L^2(\dot B^{d/2}_{2,1})}^{\ell,\wt\eps}
+\|a^\eps\|_{L^2(\dot B^{d/p}_{p,1})}^{h,\wt\eps}\bigr)\\
&&\qquad\qquad\qquad+\|\cQ u^\eps\|_{L^2(\dot B^{d/p}_{p,1})}\|\vartheta^\eps\|_{L^2(\dot B^{d/p-1}_{p,1})}^{h,\wt\eps},\end{array}
$$
$$\begin{array}{lll}
\|J(\eps a^\eps)\Delta(\vartheta^\eps)^{\ell,\wt\eps}\|_{L^1(\dot B^{(d+1)/p-3/2}_{p,1})}
&\!\!\!\!\lesssim\!\!\!\!& \|\eps a^\eps\|_{L^\infty(\dot B^{(d+1)/p-1/2}_{p,1})} \|\Delta\vartheta^\eps\|^{\ell,\wt\eps}_{L^1(\dot B^{d/2-1}_{2,1})},\\[1.5ex]
\|J(\eps a^\eps)\Delta(\vartheta^\eps)^{h,\wt\eps}\|_{L^1(\dot B^{d/p-2}_{p,1})}
&\!\!\!\!\lesssim\!\!\!\!& \|\eps a^\eps\|_{L^\infty(\dot B^{d/p}_{p,1})} \|\Delta\vartheta^\eps\|_{L^1(\dot B^{d/p-2}_{p,1})}^{h,\wt\eps},\\[1.5ex]
\bigl\|\frac\eps{1+\eps a^\eps}
\bigl(2\mu|Du^\eps|^2\!+\!\lambda(\div u^\eps)^2\bigr)\|_{L^1(\dot B^{d/p-2}_{p,1})}
&\!\!\!\!\lesssim\!\!\!\!&\eps (1\!+\!\|\eps a\|_{L^\infty(\dot B^{d/p}_{p,1})})\|Du^\eps\|_{L^2(\dot B^{d/p-1}_{p,1})}^2.
\end{array}$$
Putting all the above inequalities together, remembering of \eqref{eq:uef} and \eqref{eq:strong1}, and setting
$$\delta\!X^\eps:= \|\dT^\eps\|_{L^\infty(\dot B^{(d+1)/p-3/2}_{p,1}+\dot B^{d/p-2}_{p,1})}
+\|\dT^\eps\|_{L^2(\dot B^{(d+1)/p-1/2}_{p,1})+L^1(\dot B^{d/p}_{p,1})},
$$
 we eventually get
$$
\displaylines{\delta\! X^\eps\lesssim \|\Theta^\eps_0-\Theta_0\|_{\dot B^{(d+1)/p-3/2}_{2,1}+\dot B^{d/p-2}_{p,1}}
+\wt\eps\, C_0^{\eps,\nu}+ \nu^{-1}(1+\nu^{-1}C_0^{\eps,\nu})\wt\eps^{\,1/2-1/p}(C_0^{\eps,\nu})^2
\hfill\cr\hfill+ \|\du^\eps\|_{L^\infty(\dot B^{(d+1)/p-3/2}_{p,1}+\dot B^{d/p-2}_{p,1})}
\|\nabla\Theta\|_{L^1(\dot B^{d/p}_{p,1})}+\nu^{-1}C_0^{\eps,\nu} \delta\! X^\eps,}
$$
which allows to conclude to \eqref{eq:convtheta}.
\end{proof}


\section{Appendix} 

In this short appendix, we recall the definition of paraproduct and remainder operators, 
and give some technical estimates that have been used throughout in the paper. 
\medbreak
To start with, let us recall that, in the homogeneous setting, the paraproduct and remainder operators
$T$ and $R$ are formally defined as follows:
$$
T_uv:=\sum_{j\in\Z} \dot S_{j-1}u\, \ddj v\quad\hbox{and}\quad
R(u,v):=\sum_{j\in\Z}\ddj u \,\bigl(\dot\Delta_{j-1}\!+\!\ddj\! +\!\dot\Delta_{j+1}\bigr)v
$$
where $\dot S_k$ stands for the low-frequency cut-off operator defined by 
$\dot S_k:=\chi(2^{-k}D).$ 
\medbreak
The fundamental observation is that 
the general term of $T_uv$ is 
spectrally localized in the annulus $\bigl\{\xi\in\R^d\, ,\, 1/12\leq2^{-j}|\xi|\leq10/3\bigr\},$
and that the general term of $R(u,v)$ is localized in the ball $B(0,2^j.20/3)$
(of course the values $1/12$, $10/3$ and $20/3$ do not matter).
\medbreak
The main interest of the above definition lies in the following Bony's decomposition 
(first introduced in \cite{Bony}):
$$
uv=T_uv+ R(u,v)+T_vu,
$$
that has been used repeatedly  in the present paper.
\medbreak
The following lemma has been used to get appropriate estimates of the solution both in 
the barotropic and in the polytropic cases:
\begin{lem}\label{l:com}  Let $A(D)$ be a $0$-order Fourier multiplier, and $j_0\in\Z.$
Let $s<1,$ $\sigma\in\R$ and $1\leq p,p_1,p_2\leq\infty$ with $1/p=1/p_1+1/p_2.$
Then there exists a constant $C$ depending only on $j_0$ and
on the regularity parameters such that
$$
\|[\dot S_{j_0} A(D), T_a]b\|_{\dot B^{\sigma+s}_{p,1}}
\leq C\|\nabla a\|_{\dot B^{s-1}_{p_1,1} }\|b\|_{\dot B^\sigma_{p_2,\infty}}.
$$
In the limit case $s=1,$  we have
$$
\|[\dot S_{j_0} A(D), T_a]b\|_{\dot B^{\sigma+1}_{p,1}}
\leq C\|\nabla a\|_{L^{p_1}}\|b\|_{\dot B^\sigma_{p_2,1}}.
$$
\end{lem}
\begin{proof}
We just treat the case $s<1.$ 
By the definition of paraproduct, we have
$$
[\dot S_{j_0}A(D),T_a]b=\sum_{j\in\Z} [\dot S_{j_0}A(D), \dot S_{j-1}a]\ddj b.
$$
Using that $A(D)$ is homogeneous of degree $0$ and the properties of localization of 
operators $\dot S_k$ and $\dot\Delta_k,$ 
we  get for some smooth function $\wt\phi$ and for $j\leq j_0-4,$
 $$
 [\dot S_{j_0}A(D), \dot S_{j-1}a]\ddj b=\sum_{k\leq j-2} [\wt\phi(2^{-j}D), \dot\Delta_ka]\ddj b.
$$
Applying Lemma 2.97 of \cite{BCD} yields
\begin{equation}\label{eq:ap1}
\|[\wt\phi(2^{-j}D), \dot\Delta_ka]\ddj b \|_{L^p}\lesssim 
2^{-j}\|\dot\Delta_k a\|_{L^{p_1}}\|\ddj b\|_{L^{p_2}}.
\end{equation}
In the case where $j$ is close to $j_0$ (say $|j-j_0|\leq4$), 
one may still find some smooth function $\psi$ supported in an annulus, and such that
$$
[\dot S_{j_0}A(D), \dot S_{j-1}a]\ddj b=[\psi(2^{-j_0}D),\dot S_{j-1}a]\ddj b,
$$
which allows to get again \eqref{eq:ap1}. 
Summing up over $j$ and $k,$ and using convolution inequalities 
for series, it is easy to conclude to the desired inequality.
  \end{proof}
Finally, we recall the following composition result.  
  \begin{prop}\label{p:compo}
   Let $G$ be a smooth function defined on some open interval $I$
of $\R$ containing~$0.$ Assume that  $G(0)=0.$ 
Then for all $s>0,$ bounded interval $J\subset I,$   $1\leq m\leq\infty,$ and 
 function  $a$  valued in $J,$ the following estimates hold true:
$$
\|G(a)\|_{\dot B^{s}_{p,1}}\leq C\|a\|_{\dot B^s_{p,1}}\quad\hbox{and}\quad
\|G(a)\|_{\wt L^m(\dot B^{s}_{p,1})}\leq C\|a\|_{\wt L^m(\dot B^s_{p,1})}.
$$
\end{prop}




\begin{thebibliography}{99}

\bibitem{Al}  T. Alazard: Low Mach number limit of the full {N}avier-{S}tokes equations, {\em Arch. Ration. Mech. Anal.},
{\bf 180}(1), pages 1--73 (2006).

 \bibitem{BCD} H. Bahouri, J.-Y. Chemin and  R. Danchin: {\it Fourier Analysis and Nonlinear Partial Differential Equations,} Grundlehren der mathematischen Wissenschaften, {\bf 343}, Springer (2011).

\bibitem{Bony} J.-M. Bony: Calcul symbolique et propagation des singularit{\'e}s pour
les {\'e}quations aux d{\'e}riv{\'e}es partielles non lin{\'e}aires, {\it Annales Scientifiques de
l'{\'e}cole Normale Sup{\'e}rieure}, {\bf 14}, pages 209--246 (1981).

\bibitem{CMP}
M. Cannone, Y. Meyer and  F. Planchon: Solutions autosimilaires
des \'equations de Navier-Stokes, {\em S\'eminaire \'Equations
aux D\'eriv\'ees Partielles de l'\'Ecole Polytechnique},
1993--1994.

\bibitem{CD} F. Charve and R. Danchin: A global existence result for the compressible Navier-Stokes equations in the critical $L^p$ framework, {\em Arch. Ration. Mech. Anal.}, {\bf 198}(1), pages 233--271 (2010).

\bibitem{CH} J.-Y. Chemin: Th\'eor\`emes d'unicit\'e pour le
syst\`eme de Navier-Stokes tridimensionnel, {\em Journal d'Analyse
Math\'ematique}, {\bf 77}, pages 25--50 (1999).

\bibitem{ChD} N. Chikami and R. Danchin: 
On the well-posedness of the  full compressible Navier-Stokes system in critical Besov spaces, submitted.

\bibitem{D0} R. Danchin:  Local theory in critical spaces for 
compressible viscous and heat-conductive gases, {\it Communications in
Partial  Differential Equations}, {\bf 26},   1183--1233 (2001). 

\bibitem{D1} R. Danchin: Global existence in critical spaces for compressible
Navier-Stokes equations, {\em Inventiones Mathematicae}, {\bf 141}(3), pages 579--614 (2000).

\bibitem{D2} R. Danchin: Global existence in critical spaces for flows of compressible viscous and heat-conductive gases,
{\em  Arch. Ration. Mech. Anal.}, {\bf 160}(1), pages 1--39 (2001).

 \bibitem{D3}  R. Danchin: Zero Mach number limit in critical spaces for compressible Navier-Stokes equations,
   {\em Ann. Sci. \'{E}cole Norm. Sup.}, {\bf 35}(1), pages 27--75 (2002).

 \bibitem{D4} R. Danchin: Zero Mach Number Limit for Compressible Flows with Periodic Boundary Conditions,
  {\em American  Journal of  Mathematics}, {\bf 124}(6),  pages 1153--1219 (2002).

  \bibitem{D5} R. Danchin: {\em Fourier Analysis methods for compressible flows},
  Topics on compressible Navier-Stokes equations, \'etats de la recherche SMF, Chamb\'ery, 2012.
  
  \bibitem{D6}   R. Danchin: 
  A Lagrangian approach for the compressible Navier-Stokes equations, to appear in {\em Annales de l'Institut Fourier.}
  

\bibitem{DH} R. Danchin and L. He: The  Oberbeck-Boussinesq approximation in critical spaces, {\em Asymptotic Analysis}, {\bf 84}, pages 61--102 (2013).

\bibitem{DG} B. Desjardins and E. Grenier:  Low Mach number limit of viscous compressible flows in the whole space, {\em Proc. Roy. Soc. London Ser. A, Math. Phys. Eng. Sci.}, {\bf 455}, pages  2271--2279 (1999).

\bibitem{DGLM} B. Desjardins, E. Grenier, P.-L. Lions and N. Masmoudi:
 Incompressible limit for solutions of the isentropic
Navier-Stokes equations with Dirichlet boundary conditions,
{\it  Journal de  Math\'ematiques Pures et  Appliqu\'ees}, {\bf 78},
pages 461--471 (1999).


\bibitem{FN-book} E. Feireisl and A. Novotn{\'y}: {\em Singular limits in thermodynamics of viscous fluids},
  Advances in Mathematical Fluid Mechanics, Birkh\"auser Verlag, Basel (2009).
   
   \bibitem{FK} H. Fujita and T. Kato:
 On the Navier-Stokes initial value problem I, {\it Archive for Rational Mechanics
 and Analysis}, {\bf 16},  269-315 (1964).

\bibitem{HL} T. Hagstrom and J. Lorenz: All-time existence of classical solutions for slightly compressible flows,
{\it  SIAM J. Math. Anal.}, {\bf 29}, pages  652--672 (1998).

\bibitem{Haspot} B. Haspot: Existence of global strong solutions in critical spaces for barotropic viscous fluids,
{\em Arch. Ration. Mech. Anal.}, {\bf 202}(2), pages  427--460 (2011).
\bibitem{Haspot2} B. Haspot:  Well-posedness in critical spaces for the system of compressible Navier-Stokes in larger spaces,
{\em  Journal of  Differential Equations}, {\bf 251}(8), pages  2262--2295 (2011).
\bibitem{Haspot3} B. Haspot:
 Global existence of strong solution for shallow water system with large initial data on the irrotational part,  arXiv:1201.5456.

\bibitem{Hoff} D. Hoff: The zero-Mach limit of compressible flows,  {\em Comm. Math. Phys.}, {\bf 192}(3),
pages 543--554  (1998).

\bibitem{KM} S. Klainerman and A. Majda:  Compressible and incompressible fluids, 
{\em Comm. Pure Appl. Math.}, {\bf 35}, pages  629--651 (1982).

\bibitem{Klein} R. Klein:
 Multiple spatial scales in engineering and atmospheric low Mach number flows. {\em M2AN Math. Model. Numer. Anal.},
 {\bf 39}(3) pages  537--559 (2005).

\bibitem{KY} H. Kozono and M. Yamazaki: Semilinear heat equations and the
Navier-Stokes equations  with distributions in new function spaces
as initial data, {\em Communications in Partial Differential Equations}, {\bf 19}, pages 959--1014 (1994).

\bibitem{KLN}  H.-O. Kreiss, J. Lorenz and M.J. Naughton: Convergence of the solutions of the compressible to the solutions of the incompressible Navier-Stokes equations, {\it  Adv. Appl. Math.}, {\bf 12}, pages 187]--214 (1991).

\bibitem{Lions} P.-L.  Lions: {\em Mathematical Topics in Fluid Mechanics},
  Oxford Science Publications, Vol. 2, Compressible models, The {C}larendon {P}ress, {O}xford {U}niversity {P}ress, {N}ew-{Y}ork (1998).

\bibitem{LM} P.-L. Lions and N. Masmoudi: Une approche locale de la limite incompressible, {\em  C. R. Acad. Sci. Paris S\'er. I Math.},
{\bf 329}(5), pages  387--392 (1999).

\end{thebibliography}
\end{document}